\documentclass[10pt]{amsart}

\usepackage{dsfont, amsmath, times}
\usepackage[euler-digits]{eulervm}
\numberwithin{equation}{section}

\usepackage{hyperref}
\hypersetup{
    pdftoolbar=true,
    pdfmenubar=true,
    pdffitwindow=false,
    pdfstartview={FitH},
    pdftitle={Multiple orthogonal geodesic chords in nonconvex Riemannian disks using obstacles},
    pdfauthor={Roberto Giamb\`o, Fabio Giannoni and Paolo Piccione},
    pdfsubject={},
    pdfkeywords={},
    pdfnewwindow=true,
    colorlinks=true, %set this false for printable version
    linkcolor=blue,
    citecolor=blue,
    urlcolor=blue,
}

\makeatletter
\g@addto@macro\bfseries{\boldmath}
\makeatother

\newcommand{\dist}{\operatorname{dist}}
\newcommand{\cat}{\operatorname{cat}}
\newcommand{\Rcal}{\mathcal R}
\newcommand{\Hcal}{\mathcal H}
\newcommand{\Ncal}{\mathcal N}

\newcommand{\Dcal}{\mathcal D}

\newcommand{\Dgot}{\mathfrak D}

\newcommand{\Ddt}{\tfrac{\mathrm D}{\mathrm dt}}
\newcommand{\Dds}{\tfrac{\mathrm D}{\mathrm ds}}

\theoremstyle{plain}\newtheorem{teo}{Theorem}[section]
\theoremstyle{plain}\newtheorem{prop}[teo]{Proposition}
\theoremstyle{plain}\newtheorem{lem}[teo]{Lemma}
\theoremstyle{plain}\newtheorem{cor}[teo]{Corollary}
\theoremstyle{definition}\newtheorem{defin}[teo]{Definition}
\theoremstyle{remark}\newtheorem{rem}[teo]{Remark}
\theoremstyle{definition}
\theoremstyle{remark}
\theoremstyle{plain}

\title[Multiple OGC's in Riemannian disks]{Multiple orthogonal geodesic chords in nonconvex Riemannian disks using obstacles}

\author[R.\ Giamb\`o ,\ F.\ Giannoni, P. Piccione]{Roberto Giamb\`o, Fabio Giannoni, Paolo Piccione}
\address{\begin{tabular}{lll}
Universit\`a di Camerino & & Universidade de S\~ao Paulo \\
Scuola di Scienze e Tecnologie & & Departamento de Matem\'atica \\
Camerino (MC)  & & S\~ao Paulo, SP \\
Italy & &  Brazil\\
\emph{E-mail address}: {\tt roberto.giambo@unicam.it} & & \emph{E-mail address}: {\tt piccione@ime.usp.br}\\
\phantom{\emph{E-mail address}:} {\tt fabio.giannoni@unicam.it}
\end{tabular}
}

\date{January 22nd, 2018}

\subjclass[2000]{37C29, 37J45, 58E10}

\begin{document}

\begin{abstract}
We use nonsmooth critical point theory and the theory of geodesics with obstacle to show a multiplicity result about orthogonal geodesic chords in a Riemannian manifold (with boundary) which is homeomorphic to an $N$-disk. This applies to brake orbits in a potential well of a natural Hamiltonian system, providing a further step towards the proof of a celebrated conjecture by Seifert  \cite{seifert}.
\end{abstract}

\maketitle
%{docarmo}{docarmo}
%\norefnames
%\nocitenames

\section{Introduction}
Let $(\overline\Omega,g)$ be a compact smooth ($C^3$) Riemannian manifold with smooth ($C^2$) boundary $\partial\Omega$; denote by $\Omega=\overline\Omega\setminus\partial\Omega$. A \emph{geodesic chord} in $(\overline\Omega,g)$ is a geodesic $\gamma\colon[0,1]\to\overline\Omega$ with $\gamma(0),\gamma(1)\in\partial\Omega$ and $\gamma(s)\in \Omega$ for any $s \in ]0,1[$. An \emph{orthogonal geodesic chord} (OGC, for shortness) is a geodesic chord $\gamma$ in $(\overline\Omega,g)$ with $\dot\gamma(0)$ and $\dot\gamma(1)$ orthogonal to $\partial\Omega$ at $\gamma(0)$ and $\gamma(1)$ respectively.

In this paper, we will give a multiplicity result for OGC's in the case where $\overline\Omega$ is homeomorphic to an $N$-dimensional disk $\mathds D^N\subset\mathds R^N$; manifolds of this type will be called \emph{Riemannian $N$-disks}. When $\overline\Omega$ is a convex body of $\mathds R^N$,  then a classical result by Lusternik and Schnirelman \cite{LustSchn} gives the existence of at least $N$ segments contained in $\overline\Omega$ that meet $\partial\Omega$ orthogonally at both endpoints. This result has been generalized by Bos \cite{bos}, who proved the existence of at least $N$ orthogonal geodesic chords when $\overline\Omega$ is a convex domain of an arbitrary Riemannian $N$-manifold. The convexity assumption is crucial for these results, as it allows to use shortening methods.

Currently, the real challenge is to go beyond the convex case, and consider situations where two arbitrary points in $\Omega$ are not necessarily joined by a minimal geodesic in $\overline\Omega$, or, similarly, when a geodesic with endpoints in $\Omega$ may be somewhere tangent to $\partial\Omega$. We will think of this situation as if $\partial\Omega$ were an \emph{obstacle} to the geodesical connectedness of $\overline\Omega$, or to the existence of OGC's in $\overline\Omega$. \smallskip

The interest in this situation arises, mainly, from a famous old conjecture due to Seifert \cite{seifert} on the minimal number of periodic orbits (brake orbits)  of a Hamiltonian system with fixed energy  in a potential well. Using Maupertuis's principle, such periodic orbits can be described in terms of orthogonal geodesics chords of a Riemannian disk, whose boundary is \emph{strictly concave}, cf.\ \cite{GGP1}. 
The proof of this conjecture has motivated an intense literature on the subject (cf. e.g. \cite{rab}). In spite of the fact that a complete proof of the conjecture has not been reached yet, important contributions were given by several researchers, including Long's school \cite{LiuLong, LZ, LZZ, Long, Z1, Z2, Z3}, among many others (cf. also \cite{CV}).
\smallskip

The present paper aims at giving a further step towards a positive answer to Seifert's conjecture, with the proof of the existence of multiple OGC's in Riemannian disks with possibly non-convex boundary, but satisfying a mild global condition that can be described as follows. A geodesic segment $\gamma\colon\left[0,b\right]\to\overline\Omega$ is an \emph{O--T chord} (orthogonal-tangent) if $\gamma(0),\gamma(b)\in\partial\Omega$, $\dot\gamma(0)$ is orthogonal to $\partial\Omega$ and $\dot\gamma(b)$ is tangent to $\partial\Omega$. The central assumption of our main theorem is that $(\overline\Omega,g)$ does not admit any O--T geodesic. Two OGC's $\gamma_1,\gamma_2\colon[0,1]\to\overline\Omega$ will be considered \emph{distinct} when $\gamma_1\big([0,1]\big)\ne\gamma_2\big([0,1]\big)$. Our main result has the following statement.
 
\begin{teo}\label{thm:main}
Let $(\overline\Omega,g)$ be a Riemannian $N$-disk without O--T chords. Then, there are at least $N$ distinct orthogonal geodesic chords in $\overline\Omega$.
\end{teo}
The proof of the above results  presented here employs two techniques:
\begin{itemize}
\item a non-smooth critical point theory;
\item the theory of geodesics in manifolds with obstacles.
\end{itemize}
Note that the above result could be obtained  using a geometrical approach based on a shortening method for geodesics on manifolds with boundary. The related uniqueness result of the minimizing geodesic connecting two close points was proved in \cite{sc}.
Here, instead, we use a variational approach,  considering the geodesic action functional in the set of $H^1$-curves in $\overline\Omega$ having free endpoints in $\partial\Omega$. The standard modeling over the Hilbert space $H^1\big([0,1],\mathds R^N\big)$ does not provide a differentiable structure on such a set: paths that are somewhere tangent to  $\partial\Omega$ are singular points of this set.
The notion of criticality for this variational setup follows the theory of the \emph{weak slope} (cf. \cite{CD, CDM, degmarz}), which gives us the deformation lemmas (see Section~\ref{sec:deflem}).

The theory of geodesics with obstacle (namely curves having second derivative in $L^\infty$ satisfying equation \eqref{eq:8.1bis} below) developed in the pioneering work of Marino and Scolozzi \cite{MS} for geodesics having fixed endpoints, is used in a crucial way to show the $C^1$-regularity of the critical points, while the assumption of absence of O--T chords is used to show that critical points at a positive critical level correspond indeed to OGC's in $\overline\Omega$ (Proposition~\ref{prop:crit=OGC}).%
\smallskip

Multiplicity of gedesics in a non-contractible Riemannian manifold $M$ with boundary, connecting submanifolds with boundary $M_0$ and $M_1$, was studied in \cite{scrm}. Assuming $M_0$  and $M_1$ contractible in $M$, the authors obtain the existence of infinitely many geodesics. Such a result has been improved in \cite{canino} with weaker regularity assumptions on $M$, which is assumed only to be a \emph{$p$--convex set}. In both the above papers, the authors employ minimax methods, using the theory of curves of maximal slope, (cf. e.g. \cite{DGMT,DMT}).

A multiplicity result in an annulus  (at least two orthogonal geodesic chords) has   been proved  in \cite{arma}, together with appplications to multiplicity results for brake orbits and homoclinics. Then in \cite{JDE0} examples of potentials with exactly two brake orbits and two homoclinics are given.
\smallskip

A few remarks are in order. First, let us observe that the absence of O--T chords in $\overline\Omega$ is a condition which is open w.r.t.\ the $C^1$-topology in the set of smooth Riemannian metrics in $\overline\Omega$. It is in particular satisfied when $g$ is a metric with rotational symmetry around some point $q\in\Omega$, in which case all chords that start orthogonally from $\partial\Omega$ also arrive orthogonally onto $\partial\Omega$. Thus, our theorem guarantees the existence of at least $N$ distinct OGC's when $g$ is obtained from a small $C^1$-perturbation of a rotationally symmetric metric in $\Omega$.

It should also be observed that the absence of O--T chords is \emph{not} a generic condition. It is rather intuitive that one can construct examples of metrics on the disk admitting O--T chords, and such that the existence of O--T chords persists after sufficiently small perturbations of the metric. We will give a formal proof of this fact in Appendix~\ref{app:A} using a transversality argument. Nevertheless, also the existence of multiple OGC's is not a generic property for metrics on the disk, by the same type of arguments, so that the assumption in our Theorem is perfectly compatible with the problem.

Let us also remark that our result allows to improve the multiplicity result obtained in \cite{JDE} concerning small perturbations of the radial case.

Using the relation between orthogonal geodesic chords  and brake orbits of a natural Lagrangian (or Hamiltonian) system (cf. \cite{GGP1}),  Theorem~\ref{thm:main} gives a partial answer to the above mentioned conjecture of Seifert. Let us use the following terminology: given a Riemannian manifold $(M,g)$ and a smooth function $V\colon M\to\mathds R$, we say that the triple $(M,g,V)$ is a \emph{data for the Lagrangian system}:
\begin{equation}\label{eq:LagrSyst}
\Ddt\dot x=-\nabla V(x),
\end{equation}
where $\Ddt$ denotes the covariant derivative operator along the curve $x\colon[0,T]\to M$ relative to the Levi--Civita connection of $g$. Let us recall that brake orbits for the Lagrangian system form a special class of periodic solutions of \eqref{eq:LagrSyst}. 
\begin{teo}\label{thm:applbo}
Let $(M,g_0)$ be a Riemannian $N$-manifold whose metric is rotationally symmetric around some point $q_0 \in M$. Let $V_0\colon M\to\mathds R$ be a smooth function, which is also rotationally invariant around $q_0$, and let $E_0\in\mathds R$ be a regular value of $V_0$; assume that the sublevel $V_0^{-1}\big(\left]-\infty,E_0\right]\big)$ is homeomorphic to the $N$-disk. Then,  for any $C^2$--metric $g$ and $C^2$--function $V$ sufficiently $C^1$-close to $g_0$ and $V_0$ respectively, and for all $E\in\mathds R$ sufficiently close to $E_0$, there are at least $N$ geometrically distinct brake orbits of energy $E$ for the Lagrangian system with data $(M,g,V)$.
\end{teo}

Note that for  the above result we do not need  nondegeneracy assumptions. Observe also that the famous result of Bos (\cite{bos}) about multiplicity of OGC's in the convex case becomes a particular case of our Theorem \ref{thm:main}.\smallskip
 
\noindent\textbf{Acknowledgment.} The authors wish to thank Marco Degiovanni for useful discussions and suggestions concerning the weak slope theory and the study of Palais-Smale sequences.

\section{A nonsmooth functional framework}\label{sub:varframe}
It will be convenient to assume that $\overline\Omega$ is isometrically embedded into a larger open Riemannian $N$-manifold $(M,g)$. 
In this situation, $\Omega$ is open in $M$, its closure in $M$ coincides with $\overline\Omega$ and its boundary in $M$ is $\partial\Omega$.
Using the  Whitney Embebbing Theorem (cf \cite{wh}),  we can assume
that $M$ is an embedded submanifold of $\mathds{R}^{2N}$.  This will provide a useful linear structure to our framework. Thus, one has inclusions:
\[\overline\Omega\subset M\subset\mathds R^{2N}.\]

\subsection{Geometry of $\pmb{\overline\Omega}$}\label{sub:geoOmega}
Let us fix a $C^2$--function $\phi\colon M\to\mathds R$ with the
property that $\Omega=\phi^{-1}\big(\left]-\infty,0\right[\big)$
and $\partial\Omega=\phi^{-1}(0)$, with $\nabla\phi\ne0$ on
$\partial\Omega$. One can choose $\phi$ such that, in a neighborhood of $\partial \Omega$
\begin{equation}\label{eq:functionphi}
\phi(x)=\begin{cases}
-\dist(q,\partial\Omega),& \text{ if }x \in \overline \Omega;\\
\dist(q,\partial\Omega),& \text{ if }x \notin \overline \Omega.
\end{cases}
\end{equation}
Let us also fix $\delta_0 > 0$ small enough so that: 
\begin{equation}\label{eq:delta0}
\nabla \phi(p) \not=0 \text { for any }p \in \phi^{-1}\big([-\delta_0,\delta_0]\big);
\end{equation}
thus, the set $\phi^{-1}\big([-\delta_0,\delta_0]\big)$ is a tubular neighborhood of $\partial\Omega$, and there exists a continuous retraction:
\begin{equation}\label{eq:retraction}
\mathbf r\colon \phi^{-1}\big([-\delta_0,\delta_0]\big)\longrightarrow\partial\Omega
\end{equation}
defined in terms of the flow of $\nabla\phi$. The normalized gradient of $\phi$:
\begin{equation}\label{eq:Y}
\nu(p)=\frac{\nabla\phi(p)}{\big\Vert\nabla\phi(p)\big\Vert},\qquad p\in\partial\Omega,
\end{equation}
gives a unit normal vector field along $\partial\Omega$ which points outwards. It will be useful to consider a fixed $C^1$-extension of $\nu$ to $M$. 
The notions of convexity, or concavity, for $\overline\Omega$ can be given in terms of the second fundamental form of $\partial\Omega$ or, equivalently, in terms of the Hessian $\mathrm H^\phi$ of the function $\phi$: $\partial\Omega$ is convex (resp., concave) when for all $p\in\partial\Omega$, $\mathrm H^\phi$ is positive semidefinite (resp., negative semidefinite) on $T_p\partial\Omega$.

Let us also define the positive constant $K_0$ as follows:
\begin{equation}\label{eq:ko}
K_0=\max_{x \in \overline \Omega}\Vert \nabla \phi(x)\Vert.
\end{equation}
\subsection{Hilbert spaces and manifolds}
For $[a,b]\subset\mathds R$, let us consider the Sobolev spaces $H^1\big([a,b],\mathds R^{2N}\big)$ and $H^1_0\big([a,b],\mathds R^{2N}\big)=\big\{V\in H^1\big([a,b],\mathds R^{2N}\big):V(a)=V(b)=0\big\}$. For $S\subseteq M$, set:
\begin{equation}\label{eq:defH1abA}
H^1\big([a,b],S\big)=\big\{x\in H^1\big([a,b],\mathds R^{2N}\big):x(s)\in S\ \text{for all}\ s\in [a,b]\big\}.
\end{equation}
In particular, $H^1\big([a,b],M\big)$ is a smooth Hilbert submanifold of $H^1\big([a,b],\mathds R^{2N})$ modeled on the Hilbert space $H^1\big([a,b],\mathds R^N\big)$. For $x\in H^1\big([a,b],M\big)$, the tangent space $T_xH^1\big([a,b],M\big)$ is identified with the Hilbert subspace of $H^1\big([a,b],\mathds R^{2N}\big)$ given by:
\[T_xH^1\big([a,b],M\big)=\Big\{\xi\in H^1\big([a,b],\mathds R^{2N}\big):\xi(s)\in T_{x(s)}M\quad\text{for all $s\in[a,b]$}\Big\}.\]
We will define some distance functions on $H^1\big([0,1],M\big)$ by setting:    
\begin{multline}\label{eq:distanza}
\dist_{*}(x_1,x_2)= \max\big\{\Vert x_2(0)-x_1(0) \Vert_{2N}, \Vert x_2(1)-x_1(1)\Vert_{2N}\big\}+\\
\Big(\int_0^1 \Vert \dot x_2(s)- \dot x_1(s) \Vert_{2N}^2\,\mathrm ds\Big)^{\frac12},
\end{multline}
and
\begin{equation}\label{eq:distuniforme}
\dist_{\infty}(x_1,x_2)=\sup_{s \in [0,1]}\Vert x_2(s)-x_1(s)\Vert_{2N}.
\end{equation}
Here, $\Vert\cdot\Vert_{2N}$ is the Euclidean norm in $\mathds R^{2N}$. Note that,  
\begin{equation}\label{eq:imm}
\phantom{,\qquad \forall\,x_1,x_2 }\dist_{\infty}(x_1,x_2) \leq \dist_{*}(x_1,x_2),\qquad \forall\,x_1,x_2 \in H^1\big([0,1],M\big).
\end{equation}

Finally,  we consider the norm $\Vert\cdot\Vert_*$ in $H^1\big([0,1],\mathds R^{2N}\big)$ defined by:
\begin{equation}\label{eq:normaV}
\Vert V \Vert_{*} = \max\big\{\Vert V(0)\Vert,\Vert V(1)\Vert\big\}+
\Big(\int_a^b \big\Vert  V' \big\Vert^2\,\mathrm ds\Big)^{\frac12}.
\end{equation}
Observe that all the norms and distances defined here are invariant by backward reparameterizations of curves, see Section~\ref{sub:backwards}.
\subsection{Functional space and the energy functional}
The main functional space considered for our variational problem is: 
\begin{equation}\label{eq:defM}
\mathfrak M=\Big\{x\in
H^1\big([0,1],\overline\Omega\big):
x(0),x(1)\in\partial\Omega\Big\},
\end{equation}
endowed with the usual energy functional
\begin{equation}\label{eq:enfunct}
\mathcal F(x)=\int_0^1g(\dot x, \dot x)\,\mathrm ds.
\end{equation}
This is a smooth functional on the Hilbert manifold $H^1\big([0,1],M\big)$, and its restriction to $\mathfrak M$ is clearly continuous. For $x\in H^1\big([0,1],M\big)$ and $V\in T_xH^1\big([0,1],M\big)$, the derivative $\mathrm d\mathcal F(x)V$ is given by the integral formula:
\begin{equation}\label{eq:dermathcalF}
\mathrm d\mathcal F(x)V=2\int_0^1g\big(\dot x,\Dds V\big)\,\mathrm ds,\end{equation}
where $\Dds$ denotes the covariant derivative along $x$ relatively to the Levi--Civita connection of $g$.
\begin{lem}\label{thm:lem2.3}
Let $x\in\mathfrak M$ and $[a,b]\subset[0,1]$ be such that
$x(a)\in\partial\Omega$. Then
\begin{equation}\label{eq:aggiuntalemma2.1}
\max_{s\in[a,b]}\big|\phi(x(s)\big|\leq K_0(b-a)^\frac12 \Big[\int_a^bg(\dot x,\dot x) \,\mathrm d\sigma\Big]^{\frac12}.
\end{equation}
\end{lem}
\begin{proof}
Since $\phi(x(a))=0$ and $\Vert \nabla \phi(x) \Vert \leq K_0$ for any $x \in \overline\Omega$, for all $s\in [a,b]$ we have 
\begin{multline*}
\big\vert \phi(x(s))\big\vert = \big\vert\phi(x(s))-\phi(x(a))\big\vert\le\int_a^{s}\big\vert g\big(\nabla\phi(x(\sigma)),\dot x(\sigma)\big)\big\vert\,
\mathrm d\sigma \\
\le K_0\int_a^sg(\dot x,\dot x)^{\frac12}  \mathrm
d\sigma
\le K_0\sqrt{s-a} \left(  \int_a^bg(\dot x,\dot x) \,\mathrm
d\sigma\right)^{\frac12},
\end{multline*}
from which  \eqref{eq:aggiuntalemma2.1} follows.
%Moreover the same estimates shows that, if there exists $\bar s\in[a,b]$ such that
%$\phi(x(\bar s))\le-\delta<0$, \eqref{eq:2.9fabio} holds.
\end{proof}

\begin{cor}\label{smallstrip}
Let $x \in \mathfrak M$ and $[a,b]\subset[0,1]$ be such that $x(a),x(b)\in\partial\Omega$ and 
\[
\int_a^bg(\dot x, \dot x)\,\mathrm d\sigma \leq \frac{\delta_0^2}{K_0^2}.
\] 
Then, $\phi\big(x(s)\big)\geq -\delta_0$ for all $s \in [a,b]$.\qed
\end{cor}
\subsection{Backward reparameterization}\label{sub:backwards}
Consider now the map $\Rcal\colon \mathfrak M \to \mathfrak M$:
\begin{equation}\label{eq:2.4}
\Rcal x(t)=x(1-t).
\end{equation}
We say that $\Ncal\subset \mathfrak M$ is
\emph{$\Rcal$--invariant} if $\Rcal(\Ncal)=\Ncal$; note that
$\mathfrak M$ is $\Rcal$-invariant. If $\Ncal$ is $\mathcal
R$-invariant, a homotopy $h\colon [0,1]\times\Ncal\to\mathfrak M$ is
called \emph{$\Rcal$--equivariant} if
\begin{equation}\label{eq:2.5}
h(\tau,\Rcal x)=\Rcal h(\tau,x),\quad\forall x\in\Ncal,\,\forall\tau\in[0,1].
\end{equation}

\begin{rem}\label{rem:noROGC}
Note that if $\gamma\colon[0,1]\to\overline\Omega$ is an OGC, then 
$\gamma$ and $\Rcal\gamma$ are not distinct OGC's according to our definition, however,
$\mathcal R\gamma\not= \gamma$ as elements of $\mathfrak M$. Indeed if by contradiction $\mathcal R\gamma = \gamma$, i.e., $\gamma(1-t)=\gamma(t)$ for all $t$, and from this it would follow $\dot \gamma(\frac12)=0$. Since $g(\dot\gamma,\dot\gamma)$ is constant,  we would then  have that $\gamma$ is constant.
\end{rem}
\subsection{The set of chords of $\overline\Omega$}
The following Lemma allows to describe an $\Rcal$--invariant subset $\mathfrak C$ of ${\mathfrak M}$ which is homeomorphic to the product $\mathbb S^{N-1}\times\mathbb S^{N-1}$. This is interpreted as the set of \emph{chords} in $\overline\Omega$.
\begin{lem}\label{thm:corde} There exists  a continuous map $\gamma\colon \partial\Omega\times\partial\Omega\to H^1\big([0,1],\overline\Omega\big)$ such that
\begin{enumerate}
\item\label{corde1} $\gamma(A,B)(0)=A,\,\,\gamma(A,B)(1)=B$.
\item\label{corde2} $A\not=B\,\Rightarrow\, \gamma(A,B)(s)\in\Omega\,\forall s\in ]0,1[$.
%\item\label{corde2bis} $A\not=B\, \Rightarrow \, \gamma(A,B)(1-s)\not\equiv \gamma(A,b)(s)$.
\item\label{corde3} $\gamma(A,A)(s)=A\,\forall s\in [0,1]$.
\item\label{corde4} $\Rcal \gamma(A,B)=\gamma(B,A)$, namely $\gamma(A,B)(1-s)=\gamma(B,A)(s)$ for all $s$, and for all $A,B$.
\end{enumerate}
\end{lem}
\begin{proof}
Let $\Psi\colon  \overline \Omega  \rightarrow {\mathds D}^N$ be a  homeomorphism.  Define
\[
\hat \gamma(A,B)(s) = \Psi^{-1}\big((1-s)\Psi(A)+s\Psi(B)\big),\,A,B\in \overline\Omega.
\]
In general, if $\overline\Omega$ is only homeomorphic to the
disk $\mathds D^N$, the above definition produces curves
that in principle are only continuous. 
In order to produce curves
with an $H^1$-regularity, it suffices to use a broken geodesic approximation argument.
\end{proof}

After choosing a function $\gamma\colon \partial\Omega\times\partial\Omega\to H^1\big([0,1],\overline\Omega\big)$ as in Lemma~\ref{thm:corde}, we set:
\begin{equation}\label{eq:CC0}
\begin{aligned}
&{\mathfrak C}=\big\{\gamma(A,B)\,:\,A,B\in\partial\Omega\big\},\\
&{\mathfrak C_0}=\big\{\gamma(A,A)\,:\,A\in\partial\Omega\big\}.\end{aligned}
\end{equation}

\begin{rem}\label{rem:conseguenze-homeo}
Note  that the map $\partial\Omega\times\partial\Omega\ni(A,B)\mapsto\gamma(A,B)\in\mathfrak C$ is a homeomorphism.\end{rem}
Using the continuity of $\mathcal F$ and the compactness of $\mathfrak C$, let us define $M_0 \in \mathds{R}^+$ by:
\begin{equation}\label{eq:defM0}
M_0^2 = \sup_{x\in {\mathfrak C}}\int_0^1 g(\dot x,\dot x)\,\mathrm dt.
\end{equation}
\begin{rem}
It is interesting to observe that the constants $\delta_0$ in \eqref{eq:delta0}, $K_0$ in \eqref{eq:ko} and $M_0$ in \eqref{eq:defM0} are related by the following inequality:
\begin{equation}\label{eq:ineqM0}
M_0 > \dfrac{\delta_0}{K_0}.
\end{equation}
Indeed assume by contradiction $M_0\le\frac{\delta_0}{K_0}$.
By Corollary~\ref{smallstrip}, all the curves of $\mathfrak C$ must then lie in the tubular neighborhood $\phi^{-1}\big([-\delta_0,\delta_0]\big)$ of $\partial\Omega$. Using the retraction $\mathbf r$ defined in \eqref{eq:retraction}, one could then obtain a retraction $\mathbf r_0$ of $\partial\Omega\times\partial\Omega$ onto its diagonal, by  defining $\mathbf r_0(A,B)$ as the midpoint of the path $\mathbf r\circ\gamma(A,B)$. Of course, such a retraction cannot exist, which leads to a contradiction and proves \eqref{eq:ineqM0}.
\end{rem}

\section{On the $\mathcal V^-$-critical curves of the energy functional $\mathcal F$ on $\mathfrak{M}$}\label{sec:critical curves}
Given $x\in\mathfrak M$, let us define the notion of \emph{admissible infinitesimal variation of $x$ in $\mathfrak M$} as a vector field $V\in T_xH^1\big([0,1],M\big)$ satisfying: 
\begin{equation}\label{eq:Vbordo}
g\big(\nabla \phi(x(0)),V(0)\big)=g\big(\nabla \phi(x(1)),V(1)\big)=0,
\end{equation}
and
\begin{equation}\label{eq:Venter}
g\big(\nabla \phi(x(s)),V(s)\big)\leq 0\ \text{for any}\ s\in\left]0,1\right[\ \text{such that}\ x(s)\in\partial\Omega.
\end{equation}
Since $\nabla\phi(p)$ points outside of $\Omega$ for $p\in\partial\Omega$, then condition \eqref{eq:Venter} says that 
$V(s)$ must not point outside  $\Omega$ when $x(s)\in\partial\Omega$.
The set of admissible infinitesimal variations of $x$ in $\mathfrak M$ will be denoted by $\mathcal V^-(x)$.
\smallskip

In view of the fact that $\mathfrak M$ is not smooth, we need a suitable definition of critical points for the functional $\mathcal F\colon\mathfrak M\to\mathds R$. Following the \emph{weak slope theory}
developed in \cite{CD,degmarz}, and
recalling \eqref{eq:dermathcalF}, we give the following.
\begin{defin}\label{def:criticalcurve}
We say that $x$ is a $\mathcal V^-$-critical curve for $\mathcal F$ on $\mathfrak M$ if
\begin{equation}\label{defcc}
\int_0^1g\big(\dot x,\Dds V\big)\,\mathrm ds\geq 0,\quad \text{for all}\ V \in {\mathcal V}^-(x).
\end{equation}
\end{defin}
Note that the set $\mathcal F^{-1}(0)$ consists entirely of minimum points (the constant curves in $\partial \Omega$) which are obviously $\mathcal V^-$-critical curves. In order to describe the $\mathcal V^-$-critical curves of $\mathcal F$ corresponding to positive critical levels, we start  with the following regularity result proved in \cite{MS} in an $L^2$ setting. Here we give a simpler proof in an $H^{1}$ setting. 
\begin{prop}\label{thm:lem5.3} 
Let $z\in \mathfrak M$ be a $\mathcal V^-$-critical curve for $\mathcal F$.
Then $\Dds\dot z\in L^{\infty}\big([0,1],\mathds R^{2N}\big)$ and, in particular, $z$ is of class $C^1$. 
\end{prop}
 
\begin{proof}
Let $\nu$ be a $C^1$-vector field in $M$ as in \eqref{eq:Y}.
Given $\tau\in\mathds R$, we will write $\tau^+$ for $\max\{0,\tau\}$.
For an arbitrary $\xi\in H^1_0\big([0,1],\mathds R^{2N}\big)$ vector field along $z$, set $\zeta=\xi-\eta$, where
\begin{equation}\label{eq:eta1}
\eta(s)=g\big(\xi(s),\nu(z(s))\big)^+\cdot \nu(z(s)).
\end{equation}
Observe that $\zeta \in H^1_0\big([0,1],\mathds{R}^{2N}\big)$ and it satisfies \eqref{eq:Venter}, and therefore belongs to the set ${\mathcal V}^-(x)$ of admissible infinitesimal variations of $x$.

Let
us set 
\begin{equation}\label{eq:defCz}
C_z=\big\{s\in[0,1]\,:\,\phi(z(s))=0\big\},
\end{equation} and
$I_z=[0,1]\setminus C_z$. Since $I_z$ is open, it is a countable union of pairwise disjoint open intervals:
\[
I_z=\bigcup_{i\in J} \left]a_i,b_i\right[,
\]
where $J\subset\mathds N$. Then
\begin{equation}\label{eq:7}
\int_0^1
g\big(\dot z,\Ddt\eta\big)\,\mathrm ds=\int_{C_z}g\big(\dot z,\Ddt \eta\big)\,  \mathrm ds
+\sum_{i\in J}\int_{a_i}^{b_i}g\big(\dot z,\Ddt \eta\big)\, \mathrm ds.
\end{equation}
Let us consider an arbitrary vector field $V$ of class $H^1_0$ along
$z$ such that $V\vert_{C_z}\equiv 0$. Since both $V$ and $-V$ satisfy \eqref{eq:Vbordo} and, trivially, also \eqref{eq:Venter}, we get
\[
\int_{a_i}^{b_i}
g\big(\dot z,\Ddt \eta\big)\,  \mathrm ds=0,\quad\forall i\in J.
\]
Since $V$ is arbitrary outside $C_z$, $z_{|[a_i,b_i]}$ is a geodesic for all $i$. Partial integration gives
\begin{equation}\label{eq:5}
\sum_{i\in J}\int_{a_i}^{b_i}
g\big(\dot z,\Ddt \eta\big)\, \mathrm ds
=
\sum_{i\in J}
g\big(\dot z(b_i),\eta(b_i)\big)-g\big(\dot z(a_i),\eta(a_i)\big).
\end{equation}
But \eqref{eq:eta1} ensures that, for all $s\in C_z$,
$\eta(s)=\lambda(s)\nabla\phi\big(z(s)\big)$ with $\lambda(s)\ge 0$, while 
\[
g\big(\dot z(b_i),\nabla\phi(z(b_i)\big) \ge 0,\text{ and }
g\big(\dot z(a_i),\nabla\phi(z(a_i)\big)\le 0 \quad\forall i\in J,
\]
because
\[
\phi\big(z(s)\big) \leq 0 \text{ for all }s \in [a,b].
\] 
Then, by \eqref{eq:5},  
\begin{equation}\label{eq:6}
\sum_{i\in J}\int_{a_i}^{b_i}
g\big(\dot z,\Ddt \eta\big)\, \mathrm ds\ge 0.
\end{equation}
Since $\zeta=\xi-\eta$ satisfies 
\eqref{eq:Vbordo} and \eqref{eq:Venter} we also have 
\begin{multline}\label{eq:10a}
0\le\int_0^1
g\big(\dot z,\Ddt \xi\big)\,\mathrm ds 
-\int_{C_z}g\big(\dot z,\Ddt \eta\big)\,\mathrm  ds
-\sum_{i\in J}\int_{a_i}^{b_i}
g\big(\dot z,\Ddt \eta\big)\,\mathrm ds \\ \le
\int_0^1  g\big(\dot z,\Ddt \xi\big)\,\mathrm ds 
-\int_{C_z}  g\big(\dot z,\Ddt \eta\big)\,\mathrm ds.
\end{multline}
Since $\phi\circ z$ is of class $H^1$, and it is constant on $C_z$, we have  $g\big(\dot z,\nabla\phi(z)\big)= 0
$ a.e. on $C_z$ (cf. \cite{GT}). Then  
\[
\int_{C_z}
g\big(\dot z,\Ddt \eta\big)\,\mathrm ds
=\int_{C_z}
g\big(\dot z,\Dds{\nu}\big)\,g\big(\xi,\nu(z(s))\big)^+\,\mathrm  ds.
\]
Then by  \eqref{eq:10a} we get
\begin{equation}\label{eq:10}
\int_0^1
g\big(\dot z,\Ddt \xi\big)\,\mathrm ds
\geq
\int_{C_z}
g\big(\dot z,\Dds\nu\dot z\big)\,g(\xi,\nu)^+\,\mathrm ds.
\end{equation}
But $g(\dot z,\dot z)$ is in $L^1$, and then applying
\eqref{eq:10} to $\xi$ and $-\xi$,  we obtain the existence of a constant $L=L(z)$ such that
\[
\left\vert\int_0^1
 g\big(\dot z,\Ddt \xi\big)\,\mathrm ds
\;\right\vert
\le L(z)\,\Vert\xi\Vert_{L^\infty},
\]
for any vector field $\xi$  of class $H^1_0$, from which we deduce that $\dot z$ is in $L^\infty$.

Then, using again \eqref{eq:10} we deduce  that there exists a constant $M=M(z)$
such that
\[
\left\vert\int_0^1
 g\big(\dot z,\Ddt \xi\big)\,\mathrm  ds\;
\right\vert
\le M(z)\Vert\,\xi\Vert_{L^1}.
\]
for any vector field $\xi$ along $z$ of class $H^1_0$, from which the conclusion of the proof follows.
\end{proof}
 
\begin{lem}\label{thm:lem4.16}  
Let $z\in H^1\big([0,1],\overline \Omega\big)$ be a $\mathcal V^-$-critical curve for $\mathcal F$ in $\mathfrak M$. Then $g(\dot z, \dot z)$ is constant and 
\begin{equation}\label{eq:8.1bis}
-\Dds\dot z=
\lambda(s)\,\nu\big(z(s)\big) \text{ a.e.},
\end{equation}
where $\lambda \in L^{\infty}\big([0,1],\mathds R\big)$,  and $\lambda(s)=0$ if $\phi\big(z(s)\big)<0$.
Moreover
\begin{equation}\label{eq:lformula}
\lambda=\frac{g\big(H^{\phi}(z)[\dot z],\dot z\big)}{\sqrt{g\big(\nabla\phi(z),\nabla \phi(z)\big)}}\quad \text{a.e. in}\ C_z\quad\text{(see \eqref{eq:defCz})}
\end{equation}
and
\begin{equation}\label{eq:segno}
\lambda(s) \leq 0 \; a.e.
\end{equation}
\end{lem}
\begin{proof}
By Proposition~\ref{thm:lem5.3}, $\Dds\dot z\in L^\infty$, so we can use partial integration in \eqref{defcc}, obtaining
\begin{equation}\label{eq:crit+boundary}
\int_0^1
g\big(-\Dds\dot z,V\big)\,\mathrm ds \ge 0,\qquad\text{for
all}\ V\in {\mathcal V}^-(x)\cap H^1_0\big([0,1],\mathds R^N\big).
\end{equation}
%satisfying 
%\[g\big(\nabla\phi(z(s)),V(s)\big)\le 0\ \text{for any $s$
%such that}\ \phi(z(s))=0.
%\]
Then, $z$ is a free geodesic on any interval $I\subset[0,1]$ 
such that $z(I)\cap\partial\Omega=\emptyset$.
Moreover, using vector fields $V$ along $z$ such that $g\big(V(s),\nabla \phi(z(s))\big)=0$ for all $s$ such that $z(s)\in\partial\Omega$,  we obtain
\begin{equation}\label{eq:ast4}
\phantom{\qquad} -\Dds\dot z 
=\lambda(s)\,\nu\big(z(s)\big) \text{ a.e.},\quad\text{for some\
} \lambda\colon [0,1]\to\mathds R.
\end{equation}
Since $\Dds\dot z$ is in $L^\infty$, then $\lambda\in
L^\infty\big([0,1],\mathds R\big)$, while $\lambda=0$ a.e.\ on $z^{-1}(\Omega)$.

Now, $z$ is of class $C^1$ and $\phi\big(z(s)\big)\le 0$ for all $s\in[0,1]$. Therefore
\begin{equation}\label{eq:4.40bis}
g\big(\nabla\phi(z(s)),\dot z(s)\big)=0,\quad \text{for
every }s\in]0,1[  \text{ such that }\phi\big(z(s)\big)=0,
\end{equation}
and, contracting both terms of \eqref{eq:ast4} with $\dot z$, we obtain that $g(\dot z, \dot z)$ is constant.

Now $\phi(z(s))=0$ for all $s\in C_z$. Therefore, differentiating twice and using once again \cite{GT}, we obtain
\[
g\big(\mathrm H^{\phi}(z(s))[\dot z(s)],\dot z(s)\big)+ g\big(\nabla \phi(z(s)), \Dds\dot z(s)\big)= 0 \text{ a.e. on }C_z.
\]
Then contracting both terms of \eqref{eq:8.1bis} with $\nabla \phi(z(s)$ we obtain \eqref{eq:lformula}.\smallskip
 
It remains to prove \eqref{eq:segno}. Towards this goal,  apply formula  \eqref{eq:crit+boundary} on any interval $[\alpha,\beta] \subset [0,1]$ with
\[
V_n(s) = -\chi_n(s\,)\nabla\phi(z(s)),
\]
where $\chi_n$ is a sequence of positive bounded function, vanishing at $\alpha$ and $\beta$, and such that $\lim\limits_{n\to\infty}\chi_n(s)=1$ for all $s \in\left]\alpha, \beta\right[$.
Since $g\big(V_n,\nabla\phi(z)\big) \leq 0$ on $\partial \Omega$ for all $n$, we obtain
\[
\int_\alpha^\beta\lambda(s)\,g( \nabla\phi,\nu)\,\mathrm ds \leq 0
\]
for all $[\alpha,\beta] \subset [0,1]$. From this, we deduce  \eqref{eq:segno}.
\end{proof}
As to the endpoints of $\mathcal V^-$-critical curves, we have the following.
\begin{lem}\label{lem:orth}
If $z$ is a non-constant $\mathcal V^-$-critical curve of $\mathcal F$ in $\mathfrak M$,  then   $\dot z(0)$ and $\dot z(1)$ are orthogonal to $\partial \Omega$.
\end{lem}
\begin{proof}
Take any vector field $V$ along $z$ such that 
\[
g\big(\nabla \phi(z(s)),V(s)\big)=0, \text{ for any $s$ such that }\phi\big(z(s)\big)=0.
\]
Using again the fact that, in this case, also $-V$ is an admissible infinitesimal variation of $z$ in $\mathfrak M$, we obtain
\[
\int_0^1g(\dot z, \Dds V)\,\mathrm ds=0.
\]
Then, partial integration and \eqref{eq:8.1bis} give
\[
g\big(\dot z(1),V(1)\big)-g\big(\dot z(0),V(0)\big)=0.
\]
Since $V(1)$ and $V(0)$ are arbitrary tangent vectors to $\partial \Omega$, we obtain desired conclusion.
\end{proof}

We conclude this section with a complete description of the $\mathcal V^-$-critical curves of $\mathcal F$ in $\mathfrak M$.
\begin{prop}\label{prop:crit=OGC}
Assume that $(\overline\Omega,g)$ has no O--T chords.
If $x$ is a $\mathcal V^-$-critical curve of $\mathcal F$ in $\mathfrak M$, with $\mathcal F(x) > 0$, then $x$ is an OGC.
\end{prop}
\begin{proof}
If $\mathcal F(x)>0$, then $x$ is a non-constant $\mathcal V^-$-critical curve with $\dot x\ne0$ everywhere. By Lemma~\ref{lem:orth}, $x$ starts orthogonally from $\partial\Omega$ at the instant $s=0$, moving inside $\Omega$. We need to show that $x(s)\in\Omega$ for all $s\in\left]0,1\right[$. Assume by contradiction that this does not hold, and let $s_0\in\left]0,1\right[$ be the first instant at which $x(s_0)\in\partial\Omega$. Since $x$ is of class $C^1$ (Proposition~\ref{thm:lem5.3}), then $\dot x(s_0)$ must be tangent to $\partial\Omega$. But then $x_{|[0,s_0]}$ would be an O--T chord, which contradicts our assumption and proves the desired result.
\end{proof}

The above results entitles us to give the following:
\begin{defin}
We say that $c > 0$ is a  $\mathcal V^-$-\emph{critical value} for $\mathcal F\colon\mathfrak M\to\mathds R$ if there exists a $\mathcal V^{-}$-critical curve $x$  such that $\mathcal F(x)=c$.
\end{defin}

\section{$\mathcal V^{-}$-Palais-Smale sequences for the functional $\mathcal F$}\label{sec:limitcriticalcurves} 
In this section we  study the Palais-Smale sequences for the functional $\mathcal F$ in the strip    
${\mathcal F}^{-1}\big(\big[{\delta_0^2}/{K_0^2},M_0^2\big]\big)$.
\begin{defin}\label{def:OpenPS} 
We say that $(x_n)_n\subset \mathfrak M$
 is a $\mathcal V^-$-\emph{Palais-Smale sequence}  for $\mathcal F$ in the strip ${\mathcal F}^{-1}\big(\big[{\delta_0^2}/{K_0^2},M_0^2\big]\big)$, if 
\begin{equation}\label{eq:Fxnbounded}
 \frac{\delta_0^2}{K_0^2}\le \mathcal F(x_n)\le M_0^2,\quad\forall\,n
 \end{equation}
  and if for all (sufficiently large) $n\in\mathds N$
and for all $V_n \in \mathcal V^-(x_n)$
such that $\Vert V_n\Vert_*=1$, the following holds:
\begin{equation}\label{eq:OPS}
\mathrm d\mathcal F(x_n)[V_n] \geq -\epsilon_n,
\end{equation} 
where $\epsilon_n \to 0^+$.
\end{defin}
Note that, since $\overline\Omega$ is compact, if $(x_n)_n\subset\mathfrak{M}$ is a sequence such that $\mathcal F(x_n)$ is bounded, then $(x_n)$ admits a subsequence which is weakly convergent in $H^1$ (and therefore uniformly convergent).
Let us show that the weak limit $x$ of a $\mathcal V^{-}$-Palais-Smale sequence $(x_n)_n$ in the strip 
${\mathcal F}^{-1}\big(\big[{\delta_0^2}/{K_0^2},M_0^2\big]\big)$, is in fact a strong limit, and it is a critical point of $\mathcal F$.

\begin{prop}\label{prop:convforte}
Let $(x_n)_n$ be a $\mathcal V^{-}$-Palais-Smale sequence which is weakly convergent to a  curve $x$. Then  $(x_{n})_n$ is strongly $H^1$-convergent to $x$.
\end{prop}
 
\begin{proof} 
Let us consider an auxiliary Riemannian metric $\overline g$ in $\overline\Omega$ for which the boundary $\partial\Omega$ is totally geodesic\footnote{%
There is a standard construction of metrics for which a given closed embedded submanifold of a differentiable manifold is totally geodesic. Such metrics are constructed in a tubular neighborhood first, using a normal bundle construction, and then extended using a partition of unity argument.}, and let $\overline\exp$ denote the relative exponential map. 
Note that this means 
\begin{equation}\label{eq:TG}
\text {for all }p \in \partial \Omega, \; \overline\exp_p^{-1}(\partial \Omega)=T_p(\partial\Omega).
\end{equation}

For all $n$ sufficiently large, define the following vector field $V_n$ along the curve $x_n$:
\[V_n(s)=(\overline\exp_{x_n(s)})^{-1}\big(x(s)\big).\]
This is well defined for $n$ sufficiently large, because $x_n$ tends to $x$ uniformly. Also, $\Vert V_n\Vert_*$ is bounded, 
because $(x_n)$ is bounded in $H^1$. It is also easy to check that, using the fact that $\partial\Omega$ is totally geodesic relatively to $\overline g$,  $V_n\in\mathcal V^-(x_n)$ for all $n$. Namely, if $x_n(s)\in\partial\Omega$, since $x(s)\in\overline\Omega$, then $V_n(s)$ points inside $\Omega$. Similarly, for $s=0,1$, both $x_n(s)$ and $x(s)$ are in $\partial\Omega$, and thus $V(s)$ is tangent to $\partial\Omega$. Since $(x_n)$ is a $\mathcal V^-$-Palais-Smale sequence, by \eqref{eq:OPS} we have:
\begin{equation}\label{eq:liminfDdsVn}
\liminf_{n\to\infty}\int_0^1g\big(\dot x_n,\Dds V_n\big)\,\mathrm ds\ge0.
\end{equation}
Let $U_1,\ldots,U_v$ be the domains of local charts covering $x\big([0,1]\big)$, and let $[a_i,b_i]$ ($i=1,\ldots,k$) intervals covering $[0,1]$  such that $x\big([a_i,b_i]\big) \subset U_i$ for all $i$.
Using the fact that $V_n$ tends to $0$ uniformly, $x_n$ tends to $x$ uniformly as $n\to\infty$, and that $\Vert\dot x_n\Vert_{L^2}$ is bounded, one sees easily that, using the above local charts, in any interval $[a_i,b_i]$ the covariant derivative $\Dds V_n$ is given by an expression of the form:
\begin{equation}\label{eq:x-xn}
\Dds V_n = \dot x-\dot x_n+w_n^i,
\end{equation}
where $w_n^i$ is $L^2$-convergent to $0$. 
Moreover, by the weak $L^2$--convergence of $\dot x_n$ to $\dot x$, we have:
\[
\int_{a_i}^{b_i}g(\dot x, \dot x- \dot x_n)\,\mathrm ds \to 0.
\]
Then,
from \eqref{eq:liminfDdsVn} and \eqref{eq:x-xn} one obtains  the $H^1$-convergence of $x_n$ to $x$.
\end{proof}
In particular, the strong $H^1$-convergence implies that, if $x$ is the limit of a $\mathcal V^-$--Palais-Smale sequence for $\mathcal F$
in the strip ${\mathcal F}^{-1}\big(\big[{\delta_0^2}/{K_0^2},M_0^2\big]\big)$, then $\mathcal F(x)\ge\delta_0^2/2M_0^2$, and therefore $x$ is not a constant curve.

%In the second part of the above proof we have described, starting from some $x\in\mathfrak M$ and some $V\in\mathcal V^-(x)$, a way to construct local a extension $V_x$  of $V$, defined on some neighborhood $\mathcal B$ of $x$, satisfying $V_x(z)\in\mathcal V^-(z)$ for all $z\in\mathcal B$, see formula \eqref{eq:costr_centrale}. A similar construction will be employed once more in the rest of the paper. We give below a more general statement of this local extension property, for future reference. 

Our next step will be the proof of existence of local vector fields in $\mathfrak M$, having the property that $\mathcal F$ is decreasing along their flow lines,. We will use the following notation: for all $x \in \mathfrak M$ and all $\rho > 0$ set
\[
B(x,\rho)=\big \{z \in \mathfrak M: \Vert z-x \Vert_* < \rho\big\}
\]
and 
\begin{equation}\label{eq:defUxrho}
\mathcal U_x(\rho)=B(x,\rho) \bigcup B(\mathcal Rx,\rho),
\end{equation}
which is clearly $\mathcal R$--invariant. Let us also use the following terminology: let $x\in\mathfrak M$ and $\mu>0$ be fixed. We say that $\mathcal F$ has \emph{steepness greater than or equal to $\mu$ at $x$} if there exists $V_x \in \mathcal V^{-}(x)$ with:
\begin{itemize}
\item[(a)] $\Vert V_x \Vert_*=1$,
\item[(b)] $\displaystyle\int_0^1 g\big(\dot x, \Dds V_x\big)\,\mathrm ds\leq -\mu$.
\end{itemize}
In this situation, \emph{$V_x$ is a direction of $\mu$-steep descent} for $\mathcal F$ at $x$.
 
 For all  $\delta > 0$, we also define
\begin{equation}
\mathcal V^{-}_\delta(x) =\big\{V \in \mathcal V^{-}(x):
g(\nabla \phi(x(s)),V(s)) \leq 0\ \text{when}\  \phi(x(s))\in [-\delta,\delta]\big\}.
\end{equation}

We can now prove the following:

\begin{prop}\label{thm:constrcampiV-}
Let $x \in \mathfrak M$ be such that $\mathcal F(x) \in \left[{\delta_0^2}/{K_0^2},M_0^2\right]$ and let $\mu>0$ be fixed;  assume that $\mathcal F$ has steepness greater than or equal to $\mu$ at $x$.
Then, for any $\varepsilon > 0$ there exist 
$\rho_x=\rho_x(\varepsilon)> 0$, $\delta_x=\delta_x(\epsilon)> 0$  and a $C^{1}$-vector field $V=V_{x,\epsilon}$ defined in $\mathcal U_x(\rho_x)$,  such that:
\begin{itemize}
\item[(i)] $V(\mathcal R z)=\mathcal R V(z)$;\smallskip

\item[(ii)] $V(z) \in \mathcal V^{-}_{\delta_x}(z)$;\smallskip

\item[(iii)] $\Vert V(z) \Vert_*=1$;\smallskip

\item[(iv)] $\int_0^1 g\big(\dot z, \Dds V(z)\big)\,\mathrm ds \leq -\mu+\varepsilon$,
\end{itemize}
for all $z\in \mathcal U_x(\rho_x)$. 
\end{prop}

\begin{proof}
Assume first $x \not= \mathcal R x$. 
Let $V_x$ be a direction of $\mu$--steep  descent for $\mathcal F$ at $x$. 

Let us denote by $\overline\exp$ the exponential map of a smooth metric that makes $\partial \Omega$ totally geodesic.
For $\rho>0$ sufficiently small and $z \in B(x,\rho)$, set:
\begin{equation}\label{eq:defsuz} 
W(z)(s)=
\big(\mathrm d\,\overline\exp_{z(s)}(w(s))\big)^{-1}\big(V_x(s)\big),
\end{equation}
where $w(s)$ is defined by the relation:
\begin{equation}\label{eq:expbar}
\overline\exp_{z(s)}\big(w(s)\big)=x(s),\quad \forall\, s.
\end{equation}
From \eqref{eq:TG} it follows:
\begin{equation}\label{eq:Wbordo2}
g\big(W(z)(0),\nabla \phi(z(0))\big)=g\big(W(z)(1),\nabla \phi(z(1))\big)=0.
\end{equation}
Let $\nu$ the unit vector field in $\phi^{-1}([-\delta_0,\delta_0])$ defined by:
\[
\nu(y)=\frac{\nabla \phi(y)}{\sqrt{g(\nabla \phi(y),\nabla \phi(y))}};
\]
for $\sigma \in \mathds{R}$ denote by $\sigma^+=\max\{\sigma,0\}$:
\[
\sigma^+=\tfrac12(|s|+s),
\]
and define
\begin{multline*}
\lambda=\lambda(\rho,\delta) = \sup\Big\{g\big(W(z)(s), \nu(z(s))\big)^+:\\
 s \in [0,1], \; z \in B(x,\rho), \; \phi\big(z(s)\big)\in [-\delta,\delta]\Big\}.
\end{multline*}
Note that
\[
\lim_{(\rho,\delta)\to(0,0)}\lambda(\rho,\delta)=0.
\]
Then, for $\rho$ and $\delta$ sufficiently small, one has $\lambda \in [0,\frac12[$, and  we can define 
\[
a_\lambda=g(W(z)(\lambda),\nu(x(\lambda)))^+, \quad
b_\lambda=g(W(z)(1-\lambda),\nu(x(1-\lambda)))^+
\]
which satisfy $0\le a_\lambda, b_\lambda\le\lambda$, and therefore
\[
a_\lambda \to 0,\quad b_\lambda \longrightarrow 0\quad \text{as}\quad(\rho,\delta) \longrightarrow (0,0).
\]
Using $a_\lambda$ and $b_\lambda$, we now define the map $\chi_\lambda\colon[0,1]\to\mathds R$ by:
\begin{equation*}
\chi_\lambda(s)=
\begin{cases}
g(W(z)(s),\nu(z(s))^+,& \text{ if }s \in [0,\lambda];\\[.2cm]
\frac{b_\lambda-a_\lambda}{1-2\lambda}(s-\lambda)+a_\lambda,
& \text{ if } s \in [\lambda,1-\lambda];\\[.2cm]
g(W(z)(s),\nu(z(s))^+,& \text{ if }s \in [1-\lambda,1].
\end{cases}
\end{equation*}
Now, let us  consider piecewise linear maps $\theta,c\colon\mathds R\to[0,1]$  satisfying:
\begin{equation*}
\theta\equiv0\ \text{on }\ \left]-\infty,-\delta_0\right]\bigcup\left[\delta_0,+\infty\right[ \quad\text{and}\quad\theta\equiv1\ \text{on } \big[-\tfrac{\delta_0}2,\tfrac{\delta_0}2\big].
\end{equation*}
and 
\[
c(s)=
\begin{cases}
s,& \text{ if }s \in [0,\frac12];\\[.2cm]
1-s,& \text{ if }s \in [\frac12,1].
\end{cases}
\]
Finally, let us define 
\begin{equation}\label{eq:defWlambda}
W_\lambda(z)(s)=W(z)(s)-\big(\chi_\lambda(s)+\lambda c(s)\big)\theta\big(\phi(z(s)\big)\,\nu\big(z(s)\big).
\end{equation}
We have 
\[
g(W_{\lambda}(z)(s),\nu(z(s)) \leq 0 \text{ for all }z\in B(x,\rho),\text{and for all $s$ such that } \phi(z(s))\in [-\delta,\delta],
\]
as we can see considering separately the cases  $s \in [0,\lambda] \cup [1-\lambda,1]$ and $s \in [\lambda,1-\lambda]$. Moreover:
\[
g(W_{\lambda}(z)(0),\nu(z(0))=g(W_{\lambda}(z)(1),\nu(z(1)) =0,
\]  
which says that $W_\lambda(z) \in \mathcal V^-_{\delta}(z)$ for all $z \in B(x,\rho)$. 
It is also easy to see that

\begin{equation}\label{eq:limiteWl}
\lim_{(\rho,\delta)\to(0,0)} \sup_{z \in B(x,\rho)}
\Vert W_\lambda(z)-V_x\Vert_*=0
\end{equation}

Since for all $\rho,\delta$ sufficiently small $W_\lambda(z)\not\equiv 0$, we can define 
\[
V_\lambda(z)(s)=\frac{W_\lambda(z)(s)}{\Vert W_\lambda(z) \Vert_*}.
\]
But $\Vert V_x \Vert_* =1$, so by \eqref{eq:limiteWl}: 
\begin{equation*}
\lim_{(\rho,\delta)\to(0,0)} \sup_{z \in B(x,\rho)}
\Vert V_\lambda(z)-V_x\Vert_*=0,
\end{equation*} 
and we get property (iv).

Now, recalling that we are assuming  $x \not= \mathcal R x$, we  extend $V_\lambda$ 
to $\mathcal R(B(x,\rho_\lambda))$ by setting:
\[
V_\lambda(\mathcal Rz)=\mathcal R V_\lambda(z).
\]
Then, the desired vector field $V$ is obtained by setting $V=V_\lambda$, with $\lambda= \lambda (\rho,\delta)$, $\rho=\rho(\varepsilon)$, and $\delta=\delta(\varepsilon)$  choosen sufficiently small (depending on $\varepsilon$).

\smallskip

Now, observe that if $x=\mathcal Rx$ and   $V_x$ is a direction of $\mu$--steep descent we can consider
\[
\widetilde V_x=\frac{V_x+\mathcal R V_x }{\Vert V_x+\mathcal R V_x  \Vert_*}.
\]
Since 
\[
\int_0^1g\big(\dot x, \tfrac{\mathrm DV_x}{\mathrm ds}\big)\,\mathrm ds=\int_0^1g\big(\dot x, \tfrac{\mathrm D}{ds}\mathcal R V_x\big)\,\mathrm ds
\]
we see that $V_x+\mathcal R V_x \not\equiv 0$, and $\widetilde V_x$ is a direction of $\mu$--steep descent for $\mathcal F$. We can therefore assume that the direction $V_x$ of $\mu$--steep descent for $\mathcal F$ at $x$ is such that 
\[
V_x = \mathcal R V_x.
\]
If  $W_\lambda(z)$ is the vector field as in \eqref{eq:defWlambda}, we can finally choose 
\[ 
V(z)(s)= \frac{W_\lambda(z)(s)+ W_\lambda(\mathcal R z)(s)}{\big\Vert  W_\lambda(z)+ W_\lambda(\mathcal R z)(s) \big\Vert_*}.
\qedhere\]

\end{proof}

\section{Deformation Lemmas}\label{sec:deflem}
\subsection{Admissible homotopies}\label{sub:admhom}
Let $\mathfrak C_0$ be as in \eqref{eq:CC0}.  We will consider continuous maps
$h\colon [0,1]\times\mathfrak M\to\mathfrak M$ 
satisfying the following properties:
\begin{eqnarray}\
&h(0,\cdot) \text{ is the identity map  } 
\text{ in }\mathfrak M;& \label{eq:numero1}
\\
&h(\tau,\gamma) = \gamma \text{ for all }
\tau\in[0,1], \text{ for all } \gamma\in\mathfrak C_0;&\label{eq:numero2}
\\
&h(\tau,\cdot) \text{ is $\mathcal R$--equivariant, i.e. }
h(\tau,\mathcal R\gamma) = \mathcal Rh(\tau,\gamma) \text{ for all }\tau \in [0,1], \gamma \in \mathfrak M.&\label{eq:numero3}
\end{eqnarray}
We will call \emph{admissible homotopies} all maps $h\colon[0,1]\times\mathcal D\to\mathfrak M$ satisfying \eqref{eq:numero1}--\eqref{eq:numero3}, and we will denote by $\mathcal H$ the set of admissible homotopies.

\subsection{Admissible homotopies and weak slope}
Denote by $d$ the distance \eqref{eq:distanza}. Denote by $\mathcal H$ the class of all continuous homotopies $h$ defined 
at the above section. Let $\mathcal U_x(\rho)$ be as in \eqref{eq:defUxrho}.

According to \cite{CD,degmarz} we give the following definitions
\begin{defin}\label{def:ws}
For every $x \in \mathfrak M$ we denote by $|d\mathcal F|(x)$  the supremum of the $\sigma$'s in $\left[0,+\infty\right[$ such that there exists $\delta>0$ and a continuous homotopy $h \in \mathcal H$  such that 
\begin{equation}\label{eq:limitazione}
d\big(h(\tau,z),z\big) \leq \tau,\ 
\text{ for every }z \in \mathcal U_x(\delta),\ \text{ and }\tau\in [0,\delta],
\end{equation}
\begin{equation}\label{eq:discesa}
\mathcal F\big(h(\tau,z)\big)\leq \mathcal F(z)-\sigma\tau,\ 
\text{ for every }z \in \mathcal U_x(\delta) \text{ and }\tau\in [0,\delta].
\end{equation}
The extended real number $|d\mathcal F|(x)$ is called {\emph weak slope} of $\mathcal F$ at $x$ (with respect to $\mathcal H$).
\end{defin}

\begin{defin}\label{def:punto critico}
A curve $x \in \mathfrak M$ is called a \emph{critical curve for $\mathcal F$} if $\big|\mathrm d\mathcal F\big|(x)=0$. A real number $c$ is called critical value for $\mathcal F$ if there exists $x \in \mathfrak M$ critical curve  such that $\mathcal F(x) =c$.
\end{defin}
\begin{defin}\label{def:PS}
We say that $(x_n)_n\subset \mathfrak M$
 is a \emph{Palais-Smale sequence}  for $\mathcal F$ in the strip ${\mathcal F}^{-1}\big(\big[{\delta_0^2}/{K_0^2},M_0^2\big]\big)$, if $x_n$ satisfies \eqref{eq:Fxnbounded} and 
\begin{equation}\label{eq:infweak}
\big|\mathrm d\mathcal F\big|(x_n) \longrightarrow 0,\quad\text{as $n\to\infty$.}
\end{equation}
\end{defin}

\begin{rem}
As observed in \cite{CD}, it is an obvious consequence of its definition that the function $\vert\mathrm d\mathcal F\vert\colon \mathfrak M \to [0,+\infty]$ is lower semi-continuous. It follows that if a Palais-Smale sequence $x_n$ converges to $x$, the curve $x$ is a critical curve.
\end{rem}
Now, using the integral flow of the $C^1$-vector field $V_{x,\rho_x}$
given in Proposition \ref{thm:constrcampiV-}, we can easily adapt to our setting the the proof of \cite[Theorem 1.1.2]{CD} to obtain the following:
\begin{prop}\label{prop:flussolocale}
For every $x \in \mathfrak M$, the following inequality holds:
\[
\sup\big\{-\mathrm d\mathcal F(x)[V]:V \in \mathcal V^{-}(x),\ \Vert V \Vert_* =1\big\}\leq \big|\mathrm d\mathcal F\big|(x).\qed
\]
\end{prop}

\begin{cor}
If $x$ is a critical curve for $\mathcal F$, then $x$ is a $\mathcal V^{-}$-critical curve. Moreover, if $(x_n)$ is a Palais-Smale sequence for $\mathcal F$,  then it is also a $\mathcal V^{-}$-Palais-Smale sequence, therefore it is strongly convergent (cf Proposition \ref{prop:convforte}).\qed
\end{cor}

For $a \in \mathds{R}$, denote by $\mathcal F^a$ the closed $a$-sublevel of $\mathcal F$:
\[
\mathcal F^a=\big\{x \in \mathfrak M: \mathcal F(x) \leq a\big\}.
\]
Using the deformation results proved in \cite{CD,CDM}, we now have:
 
\begin{prop}\label{teo:deformation1}
Suppose that  $c \in\big[\delta_0^2/{K_0^2},M_0^2\big]$ is not a critical value for $\mathcal F$. Then, there exists $\varepsilon=\varepsilon(c) > 0$ and a homotopy $\eta \in \mathcal H$ such that
\[
\eta(1,\mathcal F^{c+\epsilon}) \subset \mathcal F^{c-\epsilon}.\qed
\]
\end{prop}

\subsection{Two elementary lemmas on homotopies on the sphere}
In preparation for the final result of this section, we give a statement and a short proof of two auxiliary results concerning homotopies on spheres.
Let $g$ be a Riemannian metric on $\mathbb S^m$, ($m=N-1$ for our applications), and let $\delta_g>0$ be the minimum of the injectivity radius function of $(\mathbb S^m,g)$. Here, $\mathrm{dist}$ will denote the distance function on $\mathbb S^m$ induced by $g$.
\begin{lem}\label{thm:elem1}
	Let $\mathcal C\subset\mathbb S^m\times\mathbb S^m$ be a closed set such that:
	\[0<\alpha\le\min_{(A,B)\in\mathcal C}\mathrm{dist}(A,B)<\delta_g.\]
	There exists a continuous map $\mathfrak s^\mathcal C\colon\big[0,\delta_g-\alpha\big]\times\mathcal C\to\mathbb S^m$, such that, denoting $\mathfrak s^\mathcal C(\tau,\cdot)=\mathfrak s_\tau^\mathcal C$:
	\begin{enumerate}
		\item\label{itm:pao1} $\mathfrak s^\mathcal C_0(A,B)=B$ for all $(A,B)\in\mathcal C$;
		\item\label{itm:pao2} $\tau\mapsto\mathrm{dist}\big(A,\mathfrak s^\mathcal C_\tau(A,B)\big)$ is nondecreasing for all $(A,B)\in\mathcal C$;
		\item\label{itm:pao3} $\mathfrak s^\mathcal C_\tau(A,B)=B$ for all $(A,B)\in\mathcal C$ with $\mathrm{dist}(A,B)\ge\delta_g$;
		\item\label{itm:pao4} $\min\limits_{(A,B)\in\mathcal C}\mathrm{dist}\big(A,\mathfrak s^\mathcal C_{\delta_g-\alpha}(A,B)\big)\ge\delta_g$, for all $(A,B)\in\mathcal C$.
	\end{enumerate}
\end{lem}
\begin{proof}
For all $(A,B)\in\mathcal C$ with $\mathrm{dist}(A,B)<\delta_g$, denote by $\gamma_{A,B}$ the minimal unit speed $g$-geodesic in $\mathbb S^m$ from $B$ to $A$, and with $\gamma_{A,B}(0)=B$. For $(A,B)\in\mathcal C$, one can define:
\[\mathfrak s^\mathcal C_\tau(A,B)=
\begin{cases}B,&\text{if $\mathrm{dist}(A,B)\ge\delta_g$;}\\
\gamma_{A,B}\left(\tau\cdot\frac{\mathrm{dist}(A,B)-\delta_g}{\delta_g-\alpha}\right),&\text{if $\mathrm{dist}(A,B)<\delta_g$.}
\end{cases}
\]
The continuity of $\mathfrak s^\mathcal C$ and properties \eqref{itm:pao1}--\eqref{itm:pao4} are easily verified.
	\end{proof}
The continuous map $[0,\delta_g-\alpha]\times\mathcal C\ni\big(\tau,(A,B)\big)\mapsto\big(A,\mathfrak s^\mathcal C_\tau(A,B)\big)\in\mathbb S^m\times\mathbb S^m$ will be called the \emph{endpoint separating homotopy} of the set $\mathcal C$.
	\begin{lem}\label{thm:elem2}
		Let $\mathcal C\subset\mathbb S^m\times\mathbb S^m$ be a closed set, and let $h_1\colon[0,1]\times\mathcal C\to\mathbb S^m$ be a continuous map. There exists $\tau_0>0$ such that, for all $(A,B)\in\mathcal C$ and all $E\in\mathbb S^m$ for which:\[\mathrm{dist}\big(h_1(a,(A,B)),E\big)\ge\delta_g,\quad\text{and}\quad\mathrm{dist}\big(h_1(b,(A,B)),E\big)\le\tfrac12\delta_g,\]
		for some $0\le a<b\le1$, then $b-a\ge\tau_0$.
	\end{lem}
	\begin{proof}
		It follows immediately from the uniform continuity of $h_1$ and a compactness argument.
		\end{proof}
\subsection{The deformation lemma for critical levels}\label{sub:2nddeflem}
Assume that the number of OGC's  is finite (otherwise our main result does not need a proof!), and let us denote them by $\gamma_1,\ldots,\gamma_k$. Fix $r_*> 0$ such that 
\begin{itemize}
\item the sets  $\big\{y \in \partial \Omega : \dist(y,\gamma_i(0)) < r_* \big\}$ and  $\big\{y \in \partial \Omega : \dist(y,\gamma_i(1)) < r_*\big\}$ are  contractible in $\Omega$ for all $i$ and pairwise disjoint;\smallskip

\item $B(\gamma_i,r_*) \cap B(\gamma_j,r_*) = \emptyset$ for all $i\not=j$;\smallskip

\item $B(\gamma_i,r_*) \cap B(\mathcal R \gamma_i,r_*)=\emptyset$ for all $i$ (note that $\gamma_i\ne\mathcal R\gamma_i$ for all $i$).
\end{itemize}
For every $r \in \left]0,r_*\right]$ set 
 
\begin{equation}\label{eq:defOr}
\mathcal O_r=\bigcup_{i=1}^k \mathcal U_{\gamma_i}(r),
\end{equation}
(cf \eqref{eq:defUxrho}). 
As in Proposition \ref{teo:deformation1}, using the weak slope theory we have:

\begin{prop}\label{teo:deformation2}
Let  $c \in \big[{\delta_0^2}/{k_0^2},M_0^2\big]$ be a critical value for $\mathcal F$. Then there exists $\varepsilon(r_*)>0$ and a homotopy $\eta \in \mathcal H$ such that 
\[
\eta\big(1,\mathcal F^{c+\epsilon(r_*)} \setminus \mathcal O_{r_*}\big) \subset \mathcal F^{c-\epsilon(r_*)}.\qed
\]
\end{prop}

Now let $\Dgot$ denote the family of all closed $\Rcal$--invariant subset of $\mathfrak C$. 
Given $\mathcal D\in\Dgot$ and $h \in \mathcal H$, we denote by $\mathcal A_h$ the $\mathcal R$-invariant set:
\begin{equation}\label{eq:Ah}
{\mathcal A}_h=\big\{x \in \mathcal D: h(1,x)\in \mathcal O_{r_*}\big\}.
\end{equation}

For the minimax theory we will need the following:
\begin{prop}\label{teo:deformation6}
Fix $\mathcal D\in\Dgot$, $h\in\mathcal H$, and let $\mathcal A_h$ be as in \eqref{eq:Ah}. Then, there exists a continuous map \[k_*\colon [0,1] \times \mathcal A_h \longrightarrow \mathfrak C \setminus \mathfrak C_0,\] such that:
\begin{itemize}
\item[(a)] $k_*(0,x)=x$ for every $x \in \mathcal A_h$;
\item[(b)] $k_*(\tau,\mathcal R x)=\mathcal R k_*(\tau,x)$ for all $\tau \in [0,1]$, for all $x \in \mathcal A_h$;
\item[(c)] $k_*(1,\mathcal A_h)=\{y_0,Ry_0\}$ for some $y_0 \in \mathfrak C \setminus \mathfrak C_0$.
\end{itemize}
\end{prop} 
\begin{proof}
The first observation is that a homotopy that satisfies (a), (b) and (c) above, but taking values in $\mathfrak C$ (i.e., possibly having points of $\mathfrak C_0$ in its image), does exist. This follows from the fact that, by definition, the homotopy $h$ carries $\mathcal A_h$ 
into the set $\mathcal O_{r_*}$. For $r_*>0$ small enough, this set is retractible in $\mathfrak M$ onto the finite set 
$\big\{\gamma_1,\ldots,\gamma_k,\mathcal R\gamma_1,\ldots,\mathcal R\gamma_k\big\}$, and again this finite set is retractible
in $\mathfrak M$, say, to the two-point set $\big\{\gamma_1,\mathcal R\gamma_1\big\}$. Composing with the \emph{endpoints mapping}, $\mathfrak M\ni\gamma\mapsto\big(\gamma(0),\gamma(1)\big)\in\partial\Omega\times\partial\Omega\cong\mathfrak C$, one obtains a homotopy $h_*\colon[0,1]\times\mathcal A_h\to\mathfrak C$ that carries $\mathcal A_h$ to the two-point set $\{y_0,\mathcal Ry_0\}$, where $y_0=\big(\gamma_1(0),\gamma_1(1)\big)$. 

Let us show how to use $h_*$ to construct another homotopy carrying $\mathcal A_h$ to the two-point set $\{y_0,\mathcal Ry_0\}$ in $\mathfrak C\setminus\mathfrak C_0$. Note that $\mathcal A_h\cap\mathfrak C_0=\emptyset$; namely, $h(\tau,\cdot)$ fixes all points of $\mathfrak C_0$ for all $\tau\in[0,1]$ (by definition of admissible homotopies), and $\mathcal O_{r_*}\cap\mathfrak C_0=\emptyset$ for $r_*>0$ small enough. 

The second observation is that $\mathcal A_h$ can be written as the \emph{disjoint union} $F_1\bigcup F_2$ of two closed sets $F_1,F_2\subset\mathcal D$, with $\mathcal RF_1=F_2$; namely:
\[F_1=h_*(1,\cdot)^{-1}(y_0),\quad\text{and}\quad F_2=h_*(1,\cdot)^{-1}(\mathcal R y_0).\]
To conclude the proof, it suffices to show that $F_1$ is contractible in $\mathfrak C\setminus\mathfrak C_0$ to the singleton $\{y_0\}$ via some homotopy $\widetilde h_*\colon[0,1]\times F_1\to\mathfrak C\setminus\mathfrak C_0$;  the desired homotopy $k_*$ will then be obtained by extending $\widetilde h_*$ to $[0,1]\times F_2$ by $\mathcal R$-equivariance.\smallskip
The map $\widetilde h_*$ will be constructed in two stages. First, we will find a continuous map $\ell\colon[0,1]\times F_1\to\mathfrak C\setminus\mathfrak C_0$, with $\ell\big(\tau,(A,B)\big)=\big(\ell_\tau^1(A,B),\ell_\tau^2(A,B)\big)$ for all $(A,B)\in F_1$,  such that, denoting by $\delta_g>0$ the minimum of the injectivity radius function of the metric $g$ in $\partial\Omega\cong\mathbb S^N$:
\begin{itemize}
	\item[(i)] $\ell\big(0,(A,B)\big)=(A,B)$, for all $(A,B)\in F_1$;\smallskip
	
	\item[(ii)] $\min\limits_{\stackrel{\tau\in[0,1]}{(A,B)\in F_1}}\mathrm{dist}\big(\ell_\tau^1(A,B),\ell_\tau^2(A,B)\big)\ge\frac12\delta_g$;\smallskip
	
	\item[(iii)] $\ell_{1}^1(F_1)$ is the singleton $\{\gamma_1(0)\}$.
\end{itemize}
From (i) and (ii) it follows that the set $\ell^2_{1}(F_1)$ is contained in the complement $W$ of a disk of radius $\frac12\delta_g$ centered at $\gamma(0)$ in $\mathbb S^N$. This set $W$ is contractible to the singleton $\{\gamma(1)\}$ in $\mathbb S^n$ through a homotopy that fixes all points in a disk of radius $\frac14\delta_g$ centered at $\gamma(0)$. The map $\widetilde h_*$ will be obtained by applying $\ell$ first, and then concatenating such homotopy that contracts $W$ to $\{\gamma(1)\}$.
\smallskip

Thus, we are left with the construction of the map $\ell\colon[0,1]\times F_1\to\mathfrak C\setminus\mathfrak C_0$ as described above. To this aim, we consider the map $h_*\colon[0,1]\times F_1\to\mathfrak C$ that contracts $F_1$ to the point $y_0\in\mathfrak C$; let us write $h_*\big(\tau,(A,B)\big)=\big(h^1_\tau(A,B),h^2_\tau(A,B)\big)$ for all $(A,B)\in F_1$ and $\tau\in[0,1]$. Note that 
\begin{equation}\label{eq:h1=1}
h^1_1(F_1)=\{\gamma_1(0)\}.
\end{equation}
For our construction, we will only employ the first component $h^1$ of the homotopy $h_*$; we define $\overline h_1\colon[0,1]\times F_1\to\mathfrak C\cong\mathbb S^n\times\mathbb S^n$ by:
\[\overline h_1\big(\tau,(A,B)\big)=\big(h^1_\tau(A,B),B\big).\]

The desired homotopy $\ell$ will be obtained by concatenating successively stages of the homotopy $\overline h_1$ and the ``endpoint separating homotopy'' (Lemma~\ref{thm:elem1}), as we will describe below. 
It must be observed that $\overline h_1(\tau,\cdot)$ acts only on the first component of a point $(A,B)\in F_1$, while the endpoint separating homotopy acts only on the second component of $(A,B)$.

Up to a first application of the  endpoint-separating homotopy to the set $F_1$, we can assume:
\[\min_{(A,B)\in F_1}\mathrm{dist}(A,B)\ge\delta_g.\]
Let $t_1>0$ be the smallest number in $\left]0,1\right]$ (or $t_1=1$ if no such number exists) such that $\mathrm{dist}\big(h^1_{t_1}(A,B),B\big)\le\frac12\delta_g$. On the interval $[0,t_1]$, the homotopy $\ell$ is defined to be equal to $\overline h_1$. If $t_1=1$, then our proof is concluded, by \eqref{eq:h1=1}.

Assume $t_1<1$; after this instant, we apply again the endpoint-separating homotopy to the set:
\[\Big\{\big(h^1_{t_1}(A,B),B\big):(A,B)\in F_1\Big\};\]
which is then mapped to the set $F_1'$, that can be written as:
\[F_1'=\Big\{\big(h^1_{t_1}(A,B),B_1(A,B)\big):(A,B)\in F_1\Big\}\]
for a suitable continuous map $B_1\colon F_1\to\mathfrak C$.
This set now satisfies:
\[\min_{(A_1,B_1)\in F_1'}\mathrm{dist}(A_1,B_1)\ge\delta_g,\]
and we can apply again the homotopy $\overline h_1\colon[t_1,t_2]\times F_1'\to\mathfrak C$, where $t_2$ is the smallest number in $\left]t_1,1\right]$ (or $t_2=1$ if no such number exists) such that $\mathrm{dist}\big(h_{t_2}^1(A,B),B_1(A,B)\big)\le\frac12\delta_g$. If $t_2=1$, then our proof is concluded. Otherwise, we keep repeating successive applications of the endpoint separating homotopy followed by $\overline h_1$.

The conclusion is obtained by observing that our procedure terminates after a finite number of steps, that is, at some point we will be able to apply the homotopy $\overline h_1$ until the final instant $t=1$. This follows easily from the fact that, using the uniform continuity of $\overline h_1$, the difference $t_{k+1}-t_k$ of the instants defined in our construction must be bounded from below by some positive constant $\tau_0$, say $\tau_0\ge\frac1{n_0}$, for some $ n_0 \in \mathds{N}$, see Lemma~\ref{thm:elem2} for a precise statement.
Thus, our procedure stops after at most $n_0$ steps.
\end{proof}
\section{Proof of our main Theorem}
\label{sec:mainproof}
\subsection{Relative Lusternik--Schinerlmann category}
For the minimax theory  we will use a suitable version of Lusternik--Schinerlmann relative category, as defined for instance in \cite[Definition 3.1]{FW}.
Other definitions of the relative category
and relative cohomological indexes can be found e.g.\ in \cite{FH} and the references therein.

\begin{defin}\label{def:10.1}
Let $\mathcal X$ be a topological space and $\mathcal Y$ a closed subset of $\mathcal X$. A closed subset $F$ of $\mathcal X$ has \emph{relative category} equal to $k\in\mathds N$ ($\cat_{\mathcal X,\mathcal Y}(F)=k$) if $k$ is the minimal positive integer such that $F\subset\bigcup_{i=0}^k A_i$, where  $\{A_i\}_{i=1}^k$ is a family of open subsets of $\mathcal X$ satisfying:
\begin{itemize}
\item $F\cap Y\subset A_0$;
\item for all $i=0,\ldots,k$ there exists $h_i\in C^0\big([0,1]\times
A_i,\mathcal X\big)$ with the following properties:
\begin{enumerate}
\item\label{itm:def10.1-1} $h_i(0,x)=x,\,\forall x\in A_i,\,\forall i=0,\ldots,k$; \item\label{itm:def10.1-2}
for any $i=1,\ldots,k$:
\begin{enumerate}
\item\label{itm:def10.1-2a} there exists $x_i\in\mathcal X\setminus\mathcal Y$ such that $h_i(1,A_i)=\{x_i\}$;
\item\label{itm:def10.1-2b} $h_i\colon \big([0,1]\times A_i\big)\subset\mathcal X\setminus\mathcal Y$;
\end{enumerate}
\item\label{itm:def10.1-3} if $i=0$:
\begin{enumerate}
\item\label{itm:def10.1-3a} $h_0(1,A_0)\subset\mathcal Y$; 
\item\label{itm:def10.1-3b} $h_0(\tau,A_0\cap\mathcal Y)\subset
\mathcal Y,\,\forall\,\tau\in[0,1]$.
\end{enumerate}
\end{enumerate}
\end{itemize}
\end{defin}

For any $X\subset \mathfrak M$ which is $\mathcal R$-invariant, we denote by $\widetilde X$ the quotient space with respect to the
equivalence relation induced by $\mathcal R$.

The topological invariant employed in our theory is the relative category
$\cat_{\widetilde{\mathfrak C},\widetilde{\mathfrak C}_0}(\widetilde{\mathfrak C})$.
In \cite{esistenza} it has been shown  that
\begin{equation}\label{eq:10.1}
\cat_{\widetilde{\mathfrak C},\widetilde{\mathfrak C_0}}(\widetilde{\mathfrak C})\ge N,
\end{equation}
using the topological properties of the $(N-1)$-dimensional real projective space.
\subsection{The minimax argument}
Recall that $\Dgot$ denotes the class of all closed $\Rcal$-invariant subsets of $\mathfrak C$ and that, for $\mathcal D\in\Dgot$, the symbol $\mathcal H$ denotes the set of admissible homotopies, see Section~\ref{sub:admhom}. Define, for every $i=1,\ldots,N$,
\begin{equation}\label{eq:10.2}
\Gamma_i=\big\{\Dcal\in\Dgot\,:\,\cat_{\widetilde{\mathfrak C},\widetilde{\mathfrak C}_0}(\widetilde\Dcal)\ge i\big\},
\end{equation}
and
\begin{equation}\label{eq:10.3}
c_i=\inf_{\stackrel{\Dcal\in\Gamma_i,}{h\in\mathcal H}}\sup\big\{\mathcal F(h(1,x)): x \in \mathcal D\big\}.
\end{equation}

\begin{rem}\label{rem:10.2}
Every $c_i$ is a real number. Indeed $\mathcal F \geq 0$ so $c_i\geq 0$ for all $i$. Moreover, the identity homotopy is in $\mathcal H$ and $\widetilde{\mathfrak C} \in \Gamma_i$ for any $i$, therefore
by   \eqref{eq:defM0}, 
\begin{equation}\label{eq:maggiorazione}
c_i \leq M_0^2,\quad \text{ for every }i.
\end{equation}
\end{rem}

Given continuous maps $\eta_1\colon [0,1]\times F_1\to\mathfrak M$ and $\eta_2\colon [0,1]\times F_2\to\mathfrak M$ such that $\eta_1(1,F_1)\subset F_2$,
we define the \emph{concatenation} of $\eta_1$ and $\eta_2$ as the map $\eta_2\star \eta_1\colon [0,1]\times F_1\to\mathfrak M$ given by
\begin{equation}\label{eq:defconcatenationhomotop}
\eta_2\star \eta_1(t,x)=
\begin{cases}
\eta_1(2t,x),&\text{if\ } t\in[0,\tfrac12],\\
\eta_2(2t-1,\eta_1(1,x)),&\text{if\ } t\in[\tfrac12,1].
\end{cases}
\end{equation}

The following lemmas describe the properties of the $c_i$'s.

\begin{lem}\label{thm:lem10.3} The following statements hold:
\begin{enumerate}
\item\label{itm:lem10.3-1} $c_1 \geq \dfrac{\delta_0^2}{K_0^2}$;\smallskip

\item\label{itm:lem10.3-2} $c_1\le c_2\le\cdots\le c_N$.
\end{enumerate}
\end{lem}

\begin{lem}\label{thm:lem10.4} For all $i=1,\ldots,N$,  $c_i$ is a critical value of $\mathcal F$.
\end{lem}

\begin{lem}\label{thm:lem10.5} Assume that the number of OGC's in $\overline\Omega$ is finite.
Then: 
\begin{equation}\label{eq:10.4}
c_i<c_{i+1}.
\end{equation}
for all $i=1,\ldots,N-1$.
\end{lem}

\begin{proof}[Proof of Lemma \ref{thm:lem10.3}]
Let us prove \eqref{itm:lem10.3-1}. Assume by contradiction $c_1 < \frac{\delta_0^2}{K_0^2}$,
and fix $\varepsilon>0$ so that $c_1+\epsilon < \frac{\delta_0^2}{K_0^2}$. By
\eqref{eq:10.2} and \eqref{eq:10.3} there exists $\Dcal_\varepsilon\in\Gamma_1$, and $h_\varepsilon \in\Hcal$
such that
\[
\mathcal F\big(h_\varepsilon(1,x)\big) < \frac{\delta_0^2}{K_0^2} 
,\quad
\text{ for any }x \in \Dcal_\varepsilon.
\]
Then, by Corollary~\ref{smallstrip},
\[
\inf\Big\{\phi\big(h_\varepsilon(1,x)(s)\big): s \in [0,1],\  x \in 
\Dcal_\varepsilon\Big\} \geq -\delta_0 .
\]
Using the retraction $\mathbf r$ in \eqref{eq:retraction}, we can then construct a homotopy $\Pi$ that carries $h_\varepsilon(1,\Dcal_\varepsilon)$ onto $\partial \Omega$. Finally, we define the homotopy
\[
H(\tau,x)(s)=x\big((1-\tau)s+\tfrac\tau2\big);
\]
such homotopy carries each curve $x$ to the constant curve $x(1/2)$.
We apply such $H$ to the curves of  $(\Pi \star h_\varepsilon)(1,\Dcal_\varepsilon)$, obtaining that 
$\cat_{\widetilde{\mathfrak C},\widetilde{\mathfrak C}_0}(\widetilde\Dcal_\varepsilon)=0$, in contradiction with the
definition of $\Gamma_1$.
\smallskip

To prove \eqref{itm:lem10.3-2}, fix $i\in\{1,\ldots,N-1\}$; by \eqref{eq:10.3},
for every $\varepsilon>0$ there exists $\Dcal\in\Gamma_{i+1}$ and $h\in\Hcal$ such that
\[
\mathcal F\big(h(1,x)\big)\le c_{i+1}+\varepsilon,\quad \forall\,x \in \Dcal.
\]
Since $\Gamma_{i+1}\subset\Gamma_i$. by definition of $c_i$ we deduce $c_i\le c_{i+1}+\varepsilon$, and
\eqref{itm:lem10.3-2} is proved, since $\varepsilon$ is arbitrary.
\end{proof}

\begin{proof}[Proof of Lemma \ref{thm:lem10.4}]
Assume by contradiction that $c_i$ is not a critical value for some $i$.
Take $\varepsilon=\varepsilon(c_i)$ as in Proposition~
\ref{teo:deformation1}, and choose $\Dcal_\varepsilon \in \mathfrak D$, $h\in\Hcal$ such that
\[
\mathcal F\big(h_\epsilon(1,x)\big)\le c_i+\varepsilon,\quad \text{ for all }x \in \Dcal_\varepsilon.
\]
Now let $\eta$ be as in Proposition~\ref{teo:deformation1}. Since $\eta\star h_\varepsilon \in \mathcal H$ and
\[
\mathcal F\big((\eta\star h_\varepsilon)(1,x)\big)\le c_i-\varepsilon,\quad \text{ for all }x \in \Dcal_\varepsilon,
\]
we obtain a contradiction with \eqref{eq:10.3}.  
\end{proof}

\begin{proof}[Proof of Lemma \ref{thm:lem10.5}]
Assume by contradiction that \eqref{eq:10.4} does not hold. Then there exists $i\in\{1,\ldots,N-1\}$ such that
\[
c:=c_i=c_{i+1}.
\]
Take $\varepsilon_*=\varepsilon(r_*)$ as in Proposition~\ref{teo:deformation2},
$\Dcal_{i+1}\in\Gamma_{i+1}$ and $h\in\Hcal$, such that
\[
\mathcal F\big(h(1,x)\big)\le c+\varepsilon_*,\quad \text{ for all }x \in \Dcal_{i+1}.
\]
Let
$\mathcal A_h$ be the open set defined in \eqref{eq:Ah};
 by Proposition~\ref{teo:deformation6}, its projection $\widetilde{\mathcal A_h}$ in the quotient space $\widetilde{\mathfrak C}$ is contractible in $\widetilde{\mathfrak C}\setminus\widetilde{\mathfrak C}_0$.
Then, by definition of $\Gamma_i$,
and using simple properties of the relative category, the closed set $\Dcal_i:=\Dcal_{i+1}\setminus {\mathcal A_h}$ belongs to $\Gamma_i$.
Now, let $\eta$ be as in Proposition~\ref{teo:deformation2}. We have
\[
\mathcal F\big((\eta \star h)(1,x)\big)\le c-\varepsilon_*,\quad \text{ for all }x \in \Dcal_{i},
\]
in contradiction with the definition of $\Gamma_i$. This concludes the proof.
\end{proof}

\subsection{Proof of the main theorem}
Using Lemmas \ref{thm:lem10.3}--\ref{thm:lem10.5} 
it will be sufficient to prove that if $x_1$ and $x_2$ are OGC's:
\begin{equation}\label{eq:distinzione}
x_1\big([0,1]\big)=x_2\big([0,1]\big)\quad \Longrightarrow\quad  \mathcal F(x_1)=\mathcal F(x_2).
\end{equation}
Since $\dot x_1$ and $\dot x_2$ are never vanishing, if $x_1\big([0,1]\big)=x_2\big([0,1]\big)$ we have
\begin{equation}\label{eq:ripx1x2}
x_2(s)=x_1\big(\theta(s)\big)
\end{equation}
for some $C^2$-diffeomorphism $\theta\colon[0,1]\to[0,1]$with
\begin{equation}\label{eq:theta}
\text{either}\quad\theta(0)=0,\  \theta(1)=1,\quad \text{ or }\quad \theta(0)=1,\  \theta(1)=0.
\end{equation}
Now, $x_1$ and $x_2$ satisfy the differential equation
\[
\Dds\dot x=0.
\]
Moreover, $\dot x_2(s)=\dot\theta(s)\dot x_1\big(\theta(s)\big)$ and 
\begin{equation*}
\Dds\dot x_2=\ddot\theta(s)\,\dot x_1\big(\theta(s)\big)+\big(\dot \theta(s)\big)^2\Dds\dot x_1\big(\theta(s)\big)=\\
\ddot\theta(s)\,\dot x_1\big(\theta(s)\big).
\end{equation*}
Therefore, it must be $\ddot\theta(s)\dot x_1\big(\theta(s)\big)=0$ and since 
$\dot x_1\big(\theta(s)\big)\ne0$ for all $s$, we have $\ddot\theta(s)=0$ for all $s$. Then by \eqref{eq:theta} we deduce that either $\theta(s)=s$ or $\theta(s)=1-s$: in both cases $\mathcal F(x_2)=\mathcal F(x_1)$, proving \eqref{eq:distinzione}.\qed

\appendix

\section{On the transversality of the normal bundle and the tangent bundle of submanifolds}\label{app:A}
Let $(M,g)$ be a Riemannian manifold, and let $S\subset M$ be an embedded submanifold. For the application we aim at, $S$ will be the boundary of $M$.

The normal bundle of $S$ will be denoted by $\mathcal N(S)$, and for $p\in S$, we set $\mathcal N_p(S)=\mathcal N(S)\cap T_pM$. The normal bundle $\mathcal N(S)$ and the tangent bundle $TS$ will be considered submanifolds of the tangent bundle $TM$. Given a point $p\in S$ and a normal vector $v\in\mathcal N_p(S)$, we will denote by $\alpha^S_p:T_pS\times T_pS\to\mathcal N_p(S)$ the second fundamental form of $S$ at $p$, and by $A_v^S:T_pS\to T_pS$ the shape operator of $S$ at $p$ in the direction $v$, defined by:
\[\phantom{,\quad\forall\,u,u'\in T_pS.}g\big(A_v^S(u),u'\big)=-g\big(\alpha^S(u,u'),v\big),\quad\forall\,u,u'\in T_pS.\]
For $v\in T_pM$, let us identify $T_v(TM)$ with $T_pM\oplus T_pM$, with the first summand being the horizontal space of the Levi--Civita connection of $g$ and the second summand being the vertical subspace.
\subsection{A transversality condition}
\begin{prop}\label{thm:firstcriterion}
Let $S_1,S_2\subset M$ hypersurfaces that meet transversally at  $p\in S_1\cap S_2$. Assume that $v\in T_pM\setminus\{0\}$ belongs to the intersection $\mathcal N(S_1)\cap TS_2$ and denote by $\mathcal A_v\subset T_pM$ the subspace:
\begin{equation}\label{eq:defAv}
\mathcal A_v=\big\{A_v^{S_1}(w)-\alpha^{S_2}(w,v):w\in T_pS_1\cap T_pS_2\big\}.
\end{equation}
An element $(z,z')\in T_pM\oplus T_pM\cong T_v(TM)$ belongs to the sum $T_v(\mathcal N_p(S_1)\big)+T_v(TS_2)$ iff $z'$ belongs to the affine subspace of $T_pM$:
\[\big[A^1_v(u_1)+\alpha^2_p(u_2,v)\big]+\mathcal A_v+T_pS_2,\]
where $u_1\in T_pS_1$, $u_2\in T_pS_2$ are chosen in such a way that $z=u_1+u_2$. In particular, $\mathcal N(S_1)$ and $TS_2$ meet transversally at $v$ if and only if:
\begin{equation}\label{eq:transvcondition}
\mathcal A_v\not\subset T_pS_2.
\end{equation}
\end{prop}
\begin{proof}
Using the identification $T_v(TM)\cong T_pM\oplus T_pM$, the tangent spaces $T_v\big(\mathcal N(S_1)\big)$ and $T_v(TS_2)$ are given by:
\[\begin{aligned}
&T_v\big(\mathcal N(S_1)\big)=\big\{\big(u_1,A_v^{S_1}(u_1)+\lambda\cdot v\big):u_1\in T_pS_1,\  \lambda\in\mathds R\big\},\\
&T_v(TS_2)=\big\{\big(u_2,u_2'+\alpha^{S_2}(u_2,v)\big):u_2,u_2'\in T_pS_2\big\}.
\end{aligned}\]
From these equalities, one obtains readily a proof of the first statement.

As to the transversality of  $\mathcal N(S_1)$ and $TS_2$ at $v$, i.e., $T_v\big(\mathcal N(S_1)\big)+T_v(TS_2)=T_v(TM)$, the proof follows from a linear algebra argument (apply next Lemma to $V=T_pM$, $V_1=T_pS_1$, $V_2=T_pS_2$, $\widetilde V_2=\mathds R\cdot v$, $A=A^{S_1}_v$ and $\alpha=\alpha^{S_2}(\cdot,v)$).
\end{proof}
\begin{lem}\label{thm:lemlinalg}
Let $V$ be a finite dimensional vector space, and let $V_1,V_2\subset V$ be subspaces with $V_1+V_2=V$, and $V_2$ of codimension $1$. Let $\widetilde V_2\subset V_2$ be an arbitrary subspace, let $A:V_1\to V_1$ and $\alpha:V_2\to V$ be linear maps, and define subspaces $W_1,W_2\subset V\oplus V$ and $\mathcal A\subset V$ by:
\[\begin{aligned}
&W_1=\big\{\big(u_1,A(u_1)+\widetilde v_2\big):u_1\in V_1,\ \widetilde v_2\in\widetilde V_2\big\},\\
&W_2=\big\{\big(u_2,v_2+\alpha(u_2)\big):u_2,v_2\in V_2\big\},\\
&\mathcal A=\big\{Aw-\alpha w:w\in V_1\cap V_2\big\}.\end{aligned}\]
Then, $W_1+W_2=V\oplus V$ if and only if $\mathcal A\not\subset V_2$.
\end{lem}
\begin{proof}
The equality $W_1+W_2=V\oplus V$ holds iff for all $s,t\in V$, there exist $u_1\in V_1$, $\widetilde v_2\in\widetilde V_2$ and $u_2,v_2\in V_2$ such that:
\[u_1+u_2=s,\quad\text{and}\quad A(u_1)+\widetilde v_2+v_2+\alpha(u_2)=t.\]
Since $V_1+V_2=V$, there exist $\overline u_1\in V_1$ and $\overline u_2\in V_2$ such that $\overline u_1+\overline u_2=s$; every other pair $(u_1,u_2)\in V_1\times V_2$ with $u_1+u_2=s$ is of the form $u_1=\overline u_1+h$ and $u_2=\overline u_2-h$ for some $h\in V_1\cap V_2$. Thus, the desired transversality holds iff there exists $h\in V_1\cap V_2$, $\widetilde v_2\in\widetilde V_2$ and $v_2\in V_2$ such that:
\[A(\overline u_1)+A(h)+\widetilde v_2+v_2+\alpha(\overline u_2)-\alpha(h)=t,\]
i.e.,
\[A(h)-\alpha(h)+v_2=t-\widetilde v_2-A(\overline u_1)-\alpha(\overline u_2).\]
Clearly, when $t$ is arbitrary in $V$, the right hand side of the above equality is an arbitrary vector in $V$. Thus, $W_1+W_2=V\oplus V$ holds iff
the  linear map \[(V_1\cap V_2)\times V_2\ni(h,v_2)\longmapsto A(h)-\alpha(h)+v_2\in V\] is surjective. Since $V_2$ has codimension $1$ in $V$, this holds iff the image of the linear map
$V_1\cap V_2\ni h\mapsto A(h)-\alpha(h)\in V$ is not contained in $V_2$, i.e., iff $\mathcal A\not\subset V_2$.
\end{proof}
\subsection{The case of a $1$-parameter family of submanifolds}
We will now study a transversality problem between the tangent bundle of a hypersurface $S_2$, and the union of the normal bundles of a smooth $1$-parameter family of hypersurfaces $S_1(t)$, $t\in\left]-\varepsilon,\varepsilon\right[$. \smallskip

More precisely, let us consider the following setup.
\begin{itemize}
\item $S_1,S_2\subset M$ are hypersurfaces, and $p\in S_1\cap S_2$;
\item there exists a nonzero $v\in\mathcal N_p(S_1)\cap T_pS_2$. In particular, $S_1$ and $S_2$ meet transversally at $p$.
\item $\mathbf n$ is a smooth unit vector field along $S_1$;
\item $F\colon\left]-\varepsilon,\varepsilon\right[\times S_1\to M$ is defined by:
\[F(t,x)=\exp_x\big(t\cdot\mathbf n(x)\big);\]
\item $S_1(t)$ is defined as the image of the map $F(t,\cdot)\colon S_1\to M$. 
\end{itemize}
Under suitable non-focality assumption for $p$, we can assume that $F$ is an embedding of $\left]-\varepsilon,\varepsilon\right[\times S_1$ onto an open subset of $M$. Its image is foliated by the smooth hypersurfaces $S_1(t)$.
\smallskip

We consider the union:
\[\bigcup_{t\in\left]-\varepsilon,\varepsilon\right[}\mathcal N\big(S_1(t)\big)\subset TM;\]
this is a smooth submanifold of $TM$.
Under the above assumptions, $\frac{\partial F}{\partial t}(t,x)$ is a unit vector normal to $S_1(t)$ at $F(t,x)$, and we have a parameterization of the union $\bigcup\limits_{t\in\left]-\varepsilon,\varepsilon\right[}\mathcal N\big(S_1(t)\big)$ given by the smooth map:
\[G\colon\left]-\varepsilon,\varepsilon\right[\times S_1\times\mathds R\longrightarrow TM\]
defined by:\
\[\qquad G(t,y,\lambda)=\lambda\cdot\frac{\partial F}{\partial t}(t,y),\quad (t,y,\lambda)\in\left]-\varepsilon,\varepsilon\right[\times S_1\times\mathds R.\]
Our aim is to determine when the map $G$ is transversal in $TM$ to $TS_2$ at the point $(0,p,\lambda)$, where $G(0,p,\lambda)=\lambda\cdot\mathbf n(p)=v$.

Clearly, the image of $\mathrm dG(0,p,\lambda)$ is given by:
\[\mathrm{Im}\big(\mathrm dG(0,p,\lambda)\big)=\mathds R\cdot\frac{\partial G}{\partial t}(0,p,\lambda)+T_v\big(\mathcal N(S_1)\big).\]
As we have observed above, if $T_v\big(\mathcal N(S_1)\big)+T_v(TS_2)=T_v(TM)$, i.e., if $\mathcal A_v\not\subset T_pS_2$, then $G$ is transversal to $TS_2$ at $(0,p,\lambda)$.

On the other hand, if $T_v\big(\mathcal N(S_1)\big)+T_v(TS_2)\ne T_v(TM)$, then it is easy to check that $T_v\big(\mathcal N(S_1)\big)+T_v(TS_2)$ has codimension $1$ in $T_v(TM)$. This follows immediately from the observation the the projection onto the first factor $T_v(TM)\cong T_pM\times T_pM\to T_pM$ induces an isomorphism between the quotient:
\[\left(T_v\big(\mathcal N(S_1)\big)+T_v(TS_1)\right)\big/\left(\{0\}\times T_pS_2\right)\]
and $T_pM$. 

Thus, $G$ is transversal to $TS_2$ at $(0,p,\lambda)$ iff \[\frac{\partial G}{\partial t}(0,p,\lambda)\not\in T_v\big(\mathcal N(S_2)\big)+T_v(TS_1).\]
Now, $\frac{\partial G}{\partial t}(0,p,\lambda)=\big(\mathbf n(p),0\big)$, because the $t\mapsto\frac{\partial F}{\partial t}(t,p)$ is parallel along the geodesic $t\mapsto F(t,p)$. Since $v=\lambda\cdot\mathbf n(p)$, we conclude that $G$ is transversal to $TS_2$ at $(0,p,\lambda)$ iff $(v,0)\not\in T_v\big(\mathcal N(S_1)\big)+T_v(TS_1)$. This happens iff $\alpha_p^2(v,v)\not\in T_pS_2$, i.e., iff $\alpha_p^2(v,v)\ne0$.
\medskip

In conclusion, we have proved the following:
\begin{prop}\label{thm:fintransv}
In the above notations, the submanifold \[\bigcup\limits_{t\in\left]-\varepsilon,\varepsilon\right[}\mathcal N\big(S_1(t)\big)\] is transversal to $TS_2$ at $v$ iff either one of the two conditions holds:
\begin{itemize}
\item[(a)] $\mathcal A_v\not\subset T_pS_2$;
\item[(b)] $\alpha_p^2(v,v)\ne0$.\qed
\end{itemize}
\end{prop}
\noindent Note that conditions (a) and (b) above are stable by $C^2$-small perturbations of the metric~$g$.
\begin{cor}\label{thn:nonemptyinterior}
Let $N\ge2$ and let $\mathds D^N$ be the $N$-dimensional unit disk in $\mathds R^n$. The set of Riemannian metrics on $\mathds D^N$ that admit O--T chords has nonempty interior relatively to the $C^2$-topology.
\end{cor}
\begin{proof}
Given a metric $g$ on $\mathds D^N$, an O--T chord in $\mathds D^N$ relative to $g$ corresponds to the intersection of the tangent bundle of some open portion $S_2$ of $\partial\mathds D^N$ and the normal bundle of some other open portion $S_1$ of $\partial\mathds D^N$.
Evidently, one can easily construct a (smooth) metric $g$ for which such intersection exists, and for which either one of conditions (a) and (b) of Proposition~\ref{thm:fintransv} holds. For instance, condition (b) is rather easy to obtain. In this situation, the transversality implies that sufficiently $C^2$-small perturbations of the metric will preserve the existence of some intersection between the normal and the tangent bundle of $\partial\mathds D^N$. This says that O--T chords will persist by sufficiently $C^2$-small perturbations of the metric, i.e., the set of metrics in $\mathds D^N$ admitting O--T chords has nonempty interior.
\end{proof}
It is not hard to see that, in fact, the conclusion of Corollary~\ref{thn:nonemptyinterior} holds more generally for arbitrary smooth (compact) manifolds with boundary, whose dimension is greater than or equal to $2$.

\end{document}